\documentclass[11pt,reqno]{amsproc}

\usepackage[margin=1.0in]{geometry}
\usepackage[english]{babel}
\usepackage[utf8]{inputenc}
\usepackage{ulem}
\usepackage{amsmath,amssymb,amsthm, comment, mathtools}
\usepackage{hyperref,mathrsfs,xcolor}
\usepackage{enumitem}

\mathtoolsset{showonlyrefs}

\newcommand{\R}{\mathbb{R}}

\newcommand{\n}{n} 

\renewcommand{\L}{\mathcal{L}}

\newcommand{\pa}{\partial}
\newcommand{\ve}{\varepsilon}
\newcommand{\rmd}{{\rm d}}

\usepackage[mathscr]{euscript}
\newcommand{\sL}{\mathscr{L}}
\newcommand{\sN}{\mathscr{N}}
\newcommand{\sR}{\mathscr{R}}
\newcommand{\xx}{{x_1}}
\newcommand{\xy}{{x_2}}
\newcommand{\yx}{{y_1}}
\newcommand{\yy}{{y_2}}

\usepackage{mathrsfs}

\newcommand{\eee}{equation}
\newcommand{\be}{\begin{\eee}}
\newcommand{\ee}{\end{\eee}}

\DeclareMathOperator{\curl}{curl}

\numberwithin{equation}{section}

\newtheorem{theorem}{Theorem}
\numberwithin{theorem}{section}
\newtheorem{prop}{Proposition}
\numberwithin{prop}{section}

\newtheorem{conj}{Conjecture}
\numberwithin{cor}{section}
\newtheorem{lemma}{Lemma}
\numberwithin{lemma}{section}
\theoremstyle{remark}
\newtheorem*{remark}{Remark}
\theoremstyle{definition}

\title{Flexibility and rigidity in steady fluid motion}
\author{Peter Constantin}
\address{Department of Mathematics, Princeton University, Princeton, NJ 08544}
\email{const@math.princeton.edu}
\author{Theodore D. Drivas}
\address{Department of Mathematics, Princeton University, Princeton, NJ 08544 \newline  Current address: Department of Mathematics, Stony Brook University,
Stony Brook, NY, 11794}
\email{tdrivas@math.princeton.edu, tdrivas@math.stonybrook.edu}
\author{Daniel Ginsberg}
\address{Program in Applied and Computational Mathematics, Princeton University, Princeton, NJ 08544}
\email{ dg42@princeton.edu}
\begin{document}
\date{\today}
\begin{abstract}
Flexibility and rigidity properties of steady (time-independent) solutions of the Euler, Boussinesq and Magnetohydrostatic equations are investigated.
Specifically, certain Liouville-type theorems are established which show that suitable steady solutions with no stagnation points occupying a two-dimensional periodic channel, or axisymmetric solutions in (hollowed out) cylinder,
must have certain structural symmetries.    It is additionally shown that such solutions can be deformed to occupy domains which are themselves small perturbations of the base domain.  As application of the general scheme, Arnol'd stable solutions are shown to be structurally stable.
\end{abstract}
\maketitle


\section{Introduction}

In this paper, we address two fundamental questions pertaining to steady configurations of  fluid motion (modeled here as solutions to the two-dimensional Euler and Boussinesq equations or the three-dimensional Euler equations or Magnetohydrostatic (MHS) equations).  Specifically,
\begin{itemize}
\item \textit{Rigidity:} Given a domain $D_0$ with symmetry, to what extent must steady fluid states $u_0$ conform to the symmetries of the domain?
\item  \textit{Flexibility:} Given domains $D_0, D$  which
are ``close'' in some sense, and a solution $u_0$ of a steady fluid equation
in $D_0$, can one find a steady solution $u$  in $D$ nearby $u_0$?
\end{itemize}
To study the rigidity, we show that steady fluid configurations with no stagnation points (non-vanishing velocity) confined to (topologically) annular regions have the following special property: any quantity which is steadily transported (such as vorticity in two-dimensions or temperature in the Boussinesq fluid) can be constructed as a `nice' function of the streamfunction.  This fact is then exploited by recognizing that, as a consequence, the streamfunction of such a flow must solve a certain nonlinear elliptic equation.  As such, Liouville theorems are used to constrain the possible behavior of all sufficiently regular solutions.  For the two-dimensional Euler equations on the periodic channel, this is the result of  Hamel and Nadirashvili \cite{HN17,HN19} that all steady flows without stagnation points are shears.  A similar statement can be made for a Boussinesq fluid with a certain types of stratification profiles.  For stationary three-dimensional axisymmetric Euler (e.g. pipe flow), we show that non-degenerate solutions must be purely radial for a large class of pressure profiles. That there should be a strong form of rigidity for stationary solutions of 3d Euler was already envisioned by  Harold Grad \cite{G67} who conjectured that all smooth solutions (with a certain topological structure) must conform to a symmetry.

To study the flexibility, we modify an idea introduced by Wirosoetisno and Vanneste \cite{WV05}.
To illustrate the general scheme we note that steady states $u_0$ which are tangent to the boundary can be constructed from a scalar stream function $\psi_0: D_0\to \mathbb{R}$ which solves
\begin{align}
\mathcal{L} \psi_0 &= \mathcal{N}_0(\psi_0), \quad \text{ in } D_0,\\
\psi_0 &= {\rm (const.)}  \quad \text{ on }  \pa D_0,
\end{align}
where $\mathcal{L}$ is some linear elliptic operator and $ \mathcal{N}_0$ is some nonlinear function.
We seek a solution in a nearby  domain $D$ by imposing that the stream function $\psi$ have the form
\be\label{psieqn}
\psi = \psi_0\circ \gamma^{-1},
\ee
 where $\gamma: D_0 \to D$
is a diffeomorphism to be determined. We furthermore require that $\psi$ satisfies
\begin{equation}\label{newpsi}
\mathcal{L} \psi = \mathcal{N}(\psi),  \quad \text{ in } D,
\end{equation}
for a function $\mathcal{N}= \mathcal{N}_0+\chi$ with $\chi$ conforming to the structure of the steady equations  to be determined.
Note that by construction, we automatically have  $\psi = {\rm (const.)}$ on $ \pa D$ since
$ \psi_0 = {\rm (const.)}$ on $\pa D_0$ and $\gamma: \partial D_0 \to \partial D$.
 Thus, if such a diffeomorphism can be found, $\psi$ defines a stream function for a steady solution in $D$ which is tangent to the boundary.

Having fixed the form of $\psi$ by \eqref{psieqn}, we regard \eqref{newpsi} as an equation for the diffeomorphism $\gamma$.
To solve it, we transform it into an equation in $D_0$ by composing with $\gamma$,
\begin{equation}
\mathcal{L}_\gamma \psi_0  =
  \mathcal{N}(\psi_0), \quad \text{where} \quad  \mathcal{L}_\gamma f:=\mathcal{L} \big(f\circ \gamma^{-1}\big)\circ \gamma.
 \label{maineqn}
\end{equation}
Under certain conditions on the original domain $D_0$ and the base steady state $\psi_0$, the equation \eqref{maineqn} becomes a non-degenerate, nonlinear elliptic equation for the components of the map $\gamma$ which can be solved provided the deformations are sufficiently small.  This scheme to produce a $\gamma$ has an infinite-dimensional degree of freedom which can be removed  by fixing the Jacobian of the diffeomorphism $\rho = \det \nabla \gamma$.  This steady state will solve \eqref{newpsi} potentially  with a modified nonlinearity $ \mathcal{N}$ which is completely determined in the construction of the map $\gamma$ with a given  $\rho$.

Theorem \ref{thm1} in \S \ref{flexsec} is our main result in this direction.
We illustrate its consequences in the following three cases
 \begin{itemize}
 \item
  \textit{2d Euler} for domains close to the periodic channel (see Figure \ref{fig:chan}) and Arnol'd stable steady states on compact Riemannian manifolds.
  \item
  \textit{2d Boussinesq}
  for domains close to the periodic channel (see Figure \ref{fig:chan}).
    \item
  \textit{3d axisymmetric Euler}
  for domains close to the cylinder (see Figure \ref{fig:torus}).
\end{itemize}
We now describe these settings in greater detail and state our Theorems for each case before proceeding to the proofs.

\vspace{2mm}

\begin{center}
\textbf{Two-dimensional Euler Equations}
\end{center}
Given a bounded domain $D_0\subset \mathbb{R}^2$ with smooth boundary, a  steady solution of $2d$ Euler satisfies
\begin{align}\label{ee1}
u_0 \cdot \nabla u_0 &= -\nabla p_0, \quad \text{ in } D_0, \\ \label{ee2}
\nabla \cdot u_0 &=0, \quad\quad\quad \text{ in } D_0,\\
u_0\cdot \hat{n} &=0,   \quad\quad\quad \text{ on }  \pa D_0.\label{ee3}
\end{align}
As a consequence of incompressibility, the solutions $u_0$ to the above can be constructed from a  stream function $\psi_0$ via the formula $u_0=\nabla^\perp \psi_0$ with $ \nabla^\perp=(-\partial_2,\partial_1)$.  Since $u_0 =\nabla^\perp \psi_0$, if the velocity is tangent to the boundary then $\psi_0$ must be constant along $\partial D_0$. We consider here steady solutions with additional structure, namely that the  vorticity $\omega_0$ is a Lipschitz function of the stream function $\psi_0$ through $\omega_0 = F_0(\psi_0)$.  As it turns out (see Lemma \ref{eoslemma} below), all sufficiently regular flows in  annular domains and without stagnation points have vorticity satisfying this property for some $F_0$.
 Together with $\omega_0=\nabla^\perp\cdot u_0$ this means $\psi_0$ satisfies
\begin{alignat}{2}
 \Delta \psi_0 &= F_0(\psi_0), \quad &&\text{ in } D_0,\label{gseuler1}\\
 \psi_0 &= {\rm (const.)} , \quad &&\text{ on } \pa D_0.
 \label{gseuler2}
\end{alignat}
On the other hand, clearly any solution of the above for a Lipschitz $F_0$ is the stream function of a steady solution to 2d Euler which is tangent to the boundary.

When the domain $D_0$ is a channel
\begin{align}\label{channel}
D_0 =\{ (y_1,y_2) \ | \  &y_1\in \mathbb{T},  \ y_2\in  [0,1]\},\\
\partial D_0^{\rm top}  =\{ y_2=1\},  \ \  &\ \  \partial D_0^{\rm bot}  =\{ y_2=0\},
\end{align}
solutions of the Euler equations exhibit a certain remarkable rigidity, Theorem 1.1 of \cite{HN17}\footnote{We remark that Theorem \ref{liouE}, in effect, combines the results of  \cite{HN17} and \cite{HN19}.  In the former, they establish the Theorem on an infinite strip for velocities with non-trivial inflow/outflow and in the latter they show solutions in annular domains must be radial (i.e. solutions in periodic channels must be shears).
See also the interesting complementary work \cite{GSPSY19} which establishes that solutions with \textit{single-signed} vorticity (possibly possessing stagnation points) on $\mathbb{R}^2$ must be radial.}:

\begin{theorem}[Rigidity of non-stagnant Euler flows]\label{liouE}
Let $D_0$ be a periodic channel given by \eqref{channel} and suppose that $u_0: D_0\to \mathbb{R}^2$ be a $C^2(D_0)$ solution of \eqref{ee1}--\eqref{ee3} with the property that $\inf_{D_0} u_0> 0$.  Then $u_0$ is a shear flow, namely $u_0(y_1,y_2)= (v(y_2),0)$ for some scalar function $v(y_2)$.
\end{theorem}
Theorem \ref{liouE} is an example of a Liouville theorem for solutions of the incompressible Euler equations.  It shows that  \textit{any} smooth steady solution of the Euler equations in the channel which never vanishes must be a shear flow, isolating such configurations from non-shear steady states. 
It should be emphasized that there are many examples of non-trivial flows with stagnation points which are non-shear (e.g. cellular flows).  {In fact, Lin and Zeng \cite{LZ11} shows that there exist Cat's--eye vortices arbitrarily close to Couette flow $u_0(y)=(y,0)$ in the $H^s$, $s<3/2$ topology.  }
 Similar results to Theorem \ref{liouE} hold when the domain is the annulus or the disk (under some additional conditions on the solution) \cite{HN19}. We remark that the very interesting recent work of Coti Zelati, Elgindi, and Widmayer \cite{CEW20} shows that the assumption of non-degeneracy is not always necessary for such a Liouville theorem by establishing a similar rigidity of 2d Euler solutions near Poiseuille flow $u_0(y)=(y^2-c,0)$ for $c\geq 0$ which stagnates at $y=\pm \sqrt{c}$.

  In light of the rigidity result of \cite{HN17}, we show in Theorem \ref{2dEthm} that we can perturb away from \textit{any non-vanishing solution} of the two-dimensional Euler equations in the channel $u_0$ to domains
   \begin{align}   \label{domain1}
   D = \{ (x_1,x_2) \ | \ x_1 \in  \mathbb{T},\  b_0&(x_1) \leq x_2 \leq 1 + b_1(x_1)\},\\
\partial D^{\rm top}  =\{ x_2=1 + b_1(x_1) \},  \ \  &\ \  \partial D^{\rm bot}  =\{ x_2=b_0(x_1)\},
\end{align}
   for suitably small $b_0, b_1$ (see Figure \ref{fig:chan}) and obtain a steady solution $u$ in $D$:
   \begin{align}\label{ee1f}
u \cdot \nabla u &= -\nabla p, \quad \text{ in } D, \\ \label{ee2f}
\nabla \cdot u &=0, \quad\quad\quad \text{ in } D,\\
u\cdot \hat{n} &=0,   \quad\quad\quad \text{ on }  \pa D.\label{ee3f}
\end{align}

\vspace{-4mm}
\begin{figure}[h!]
  \includegraphics[height=1.5in]{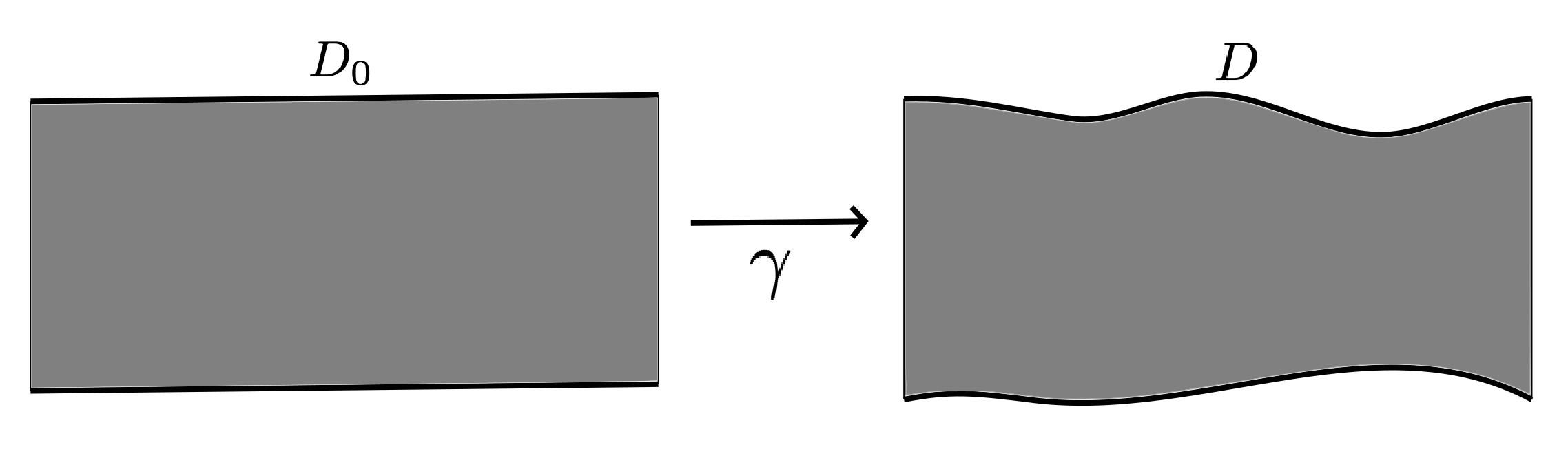}
  \caption{Deformation of periodic channel $D_0$ by $\gamma$.}
  \label{fig:chan}
\end{figure}

\begin{theorem}[Flexibility of non-stagnant Euler flows]\label{2dEthm}
  \label{existence}
Let $D_0$ be defined by \eqref{channel} and  $D$ be defined by \eqref{domain1}.
  Suppose $\psi_0:D_0\to \R$ with $U_0:= \inf_{D_0} |\nabla \psi_0|>0$ and
  $\psi_0 \in C^{k,\alpha}(D_0)$ for some $\alpha > 0, k \geq 3$ such that $u_0= \nabla^\perp \psi_0$ satisfies \eqref{ee2}--\eqref{ee3}.  Then there is an $F_0\in C^{k-2,\alpha}(\mathbb{R})$ such that $\psi_0$ satisfies \eqref{gseuler1}--\eqref{gseuler2}  and constants $\ve_1,\ve_2$ depending only on $U_0, D_0, F_0$ and
$\|\psi_0\|_{C^{k,\alpha}}$ such that if $b_0, b_1:\R \to \R$ and $\rho:D_0\to \mathbb{R}$ with $\int_{D_0} \rho = {\rm Vol}(D)$ satisfy
\begin{align}
\|b_0\|_{C^{k,\alpha}(\R)}
  + \|b_1\|_{C^{k,\alpha}(\R)} &\leq \ve_1,
\\
  \|1-\rho\|_{C^{k,\alpha}(D_0)}&\leq \ve_2,
  \end{align}
there is a diffeomorphism $\gamma : D_0 \to D$ with Jacobian ${\rm det} (\nabla \gamma)= \rho$, and a function $F: \R \to \R$ close in $C^{k-2,\alpha}$ to $F_0$ in
so that $\psi = \psi_0\circ \gamma^{-1} \in C^{k,\alpha}(D)$ and $\psi$ satisfies $\Delta \psi= F(\psi)$.
Thus, $u =\nabla^\perp \psi$  is an Euler solution in $D$ nearby $u_0$.
\end{theorem}

This theorem is a  generalization of Theorem 1 of Wirosoetisno and Vanneste \cite{WV05} to include non-volume preserving maps $\gamma$ and follows from a much more general theorem which we prove,
Theorem \ref{thm1}.  The freedom of choosing the Jacobian of the map gives an additional mechanism to reach nearby other steady states.
When $b_0$ and $b_1$ are zero, the perturbed domain is again a channel and the solution \textit{must} be a shear flow, which is a consequence of the Theorem of \cite{HN17} discussed above. Nevertheless, due to the fact that the Jacobian is an arbitrary function near unity, our procedure picks out different solutions, allowing us to ``slide" along the space of shear flows in the channel.
Note also that radial solutions on the annulus or the disk can also be deformed. In the case of a simply connected domain such as the disk, the base state must have a stagnation point and
Theorem \ref{2dEthm} applies provided that  $\psi_0$ satisfies the non-degeneracy Hypothesis \textbf{(H2)} below.
{We finally we make some remarks about the case where Hypothesis \textbf{(H2)} is violated.  In this case, if one has that the linearized operator $\Delta - F_0'(\psi_0)$ has a trivial kernel, then a standard implicit function theorem argument produces `nearby' stationary states for `nearby' vorticity profiles $F$.  This condition is implied by Arnol'd stability (see discussion below) and it holds also, for example, for Couette flow whose vorticity is constant so that $F_0'=0$.  On the other hand, this argument does not give much information on the structure of the solution, whereas Theorem \ref{2dEthm} and the other below allow one to understand and control the geometry of the streamlines to a certain extent.}

A different class of flows which display a remarkably general form of rigidity and flexibility on any domains with a symmetry are so-called Arnol'd stable steady states. Recall that a stationary state  on a domain $\Omega\subset \mathbb{R}^2$ is called Arnol'd stable if the vorticity of an Euler solution $\omega= F(\psi)$ satisfies either of the following two conditions
\begin{equation}
-\lambda_1<F'(\psi) <0, \qquad \text{or} \qquad 0< F'(\psi)<\infty
 \label{arnoldscond}
\end{equation}
where $\lambda_1:=\lambda_1(\Omega)>0$ is the smallest eigenvalue of $-\Delta$ in $\Omega$.
See \cite{A69} or Theorem 4.3 of \cite{AK99}.  The above two ranges are referred to type I and type II Arnol'd stability conditions. These conditions ensure that the steady state is either a minimum or a maximum
of the energy (the action) for a fixed vorticity distribution and guarantee that such states are orbitally stable in the topology of $L^2(\Omega)$ of the vorticity.

To emphasize the generality of what follows,  let $(M,g)$ be a two-dimensional Riemannian manifold with smooth boundary $\partial M$ and let $\xi$ be a Killing field for $g$.  Suppose that $\xi$ is tangent to $\partial M$. Consider a solution $\psi$ of
\begin{align}\label{Euler1}
\Delta_g \psi &= F(\psi), \qquad \text{in} \ M,\\
 \psi &= {\rm (const.)} , \ \ \ \text{on} \ \partial M, \label{Euler2}
\end{align}
 where $\Delta_g$ is the Laplace-Beltrami operator on $M$. The velocity  constructed from $\psi$ by
 \be
 u^i = - \sqrt{\det g} g^{ij}\epsilon_{jk} g^{kl}\frac{\partial \psi}{\partial x^l}=: \nabla_g^\perp \psi
 \ee
 is automatically a solution of the Euler equation on $M$ with vorticity $\omega= \frac{\epsilon_{ij}}{\sqrt{\det g}} \frac{\partial}{\partial x^i} (g_{jk} u^k) =F(\psi)$. In the language of
 differential forms $u = (*_g d \psi)^\sharp$ where $*_g$ and $\sharp$ denote
 the Hodge star and musical isomorphism associated to the metric $g$ and $d$
 denotes the differential.
   Since $\xi$ is Killing for $g$, the commutator of the Lie derivative  $\L_\xi $ with the Laplace-Beltrami operator vanishes $[\L_\xi , \Delta_g]=0$ (applied to tensors of any rank).   Moreover, since $\xi$ is tangent to $\partial \Omega$, on which $\psi$ takes constant values, we have that $\L_\xi\psi=0$ on $\partial \Omega$. Thus, differentiating \eqref{Euler1}--\eqref{Euler2} we obtain the equation
\begin{align}\label{p1}
\Big(\Delta_g - F'(\psi) \Big) \L_\xi \psi &= 0, \qquad \text{in} \ M,\\
 \L_\xi \psi &=0  , \ \ \ \ \  \text{on} \ \partial M . \label{p2}
\end{align}
Clearly if the operator $ \Delta_g - F'(\psi) $ has a trivial kernel in $H^1_0$, then $\L_\xi \psi =0$.
A sufficient condition to ensure this  is that $F'(\psi)> -\lambda_1$ where $\lambda_1$ is the first eigenvalue of $-\Delta_g$ on $M$.  Both type I and type II Arnol'd stability conditions ensure this.
Thus, we obtain
\begin{prop}[Rigidity of Arnol'd stable states]\label{propA}
Let $(M,g)$ be a compact two-dimensional Riemannian manifold with smooth boundary $\partial M$ and let $\xi$ be a Killing field for $g$.  Suppose that $\xi$ is tangent to $\partial M$.  Let $u: M\to \mathbb{R}^2$ be a $u= \nabla_g^\perp \psi \in C^2(M)$ Arnol'd stable state.  Then $\L_\xi \psi=0$.
\end{prop}
In some simple cases, Proposition \ref{propA} implies
\begin{itemize}
\item on the periodic channel with $\xi=e_{y_1}$ and $\zeta=e_{y_2}$, all Arnol'd stable stationary solutions are shears  $u= v(y_2) e_{y_1}$.
\item on the disk (or annulus), with $\xi=e_\theta$ and $\zeta=e_r$,  all Arnol'd stable stationary solutions are radial $u= v(r) e_\theta$.
\item on a spherical cap\footnote{On the cap of a sphere of radius $R$, we use spherical
coordinates $x = (R, \lambda, \phi)$, where $\lambda \in [-\pi, \pi]$ is longitude and $\phi \in [-\pi/2, \pi/2]$ is
latitude, with the poles at $\phi = \pm \pi/2$. } with $\xi=e_\lambda$ and $\zeta=e_\phi$,     all Arnol'd stable stationary solutions are zonal (functions of latitude) $u= v(\phi) e_\lambda$.
\item on manifolds without boundary possessing two transverse Killing fields (e.g. the two-torus or the sphere), there can be no Arnol'd stable steady states (see \cite{WS00}).
\end{itemize}
We remark that in fact the statement for the periodic channel or annulus hold whether or not the
state is Arnol'd stable, see \cite{HN17},\cite{HN19}, provided that the flow has no stagnation points.

Thus,  Proposition \ref{propA} reveals a strong form of rigidity of Arnol'd stable steady states.  However, we also show that are also flexible in the sense that nearby stable steady states exist on wrinkled domains (slight changes of the background metric) with wiggled boundaries.
\\

\vspace{-2mm}
Specifically, consider a steady solution on $M_0\subset \mathbb{R}^2$ satisfying the following hypotheses:
  \begin{itemize}
 \item [\textbf{(H1)}] The vorticity $\omega_0=F_0(\psi_0)$ satisfies
 $F_0'(\psi_0)> -\lambda_1(M_0)$. \vspace{2mm}
  \item [\textbf{(H2)}]
There is a constant $c_{\psi_0}>0$ such that for all $c\in {\rm im}(\psi_0)$ we have
   \be\label{rotation}
  \oint_{\{\psi_0=c\}} \frac{\rmd \ell}{|\nabla \psi_0|} \leq \frac{1}{c_{\psi_0}}.
  \ee

  \end{itemize}

Hypothesis \textbf{(H1)} is ensured for type I and II Arnol'd states by \eqref{arnoldscond}.
Hypothesis \textbf{(H2)}   ensures that the period of rotation of fluid parcels along streamlines (left-hand-side of \eqref{rotation}) is bounded and  is automatically satisfied for any base flow without stagnation points on  annular domains and it holds on simply connected domains if there is non-vanishing vorticity $F_0\neq 0$ at the critical point. We call flows satisfying \textbf{(H2)} \textit{non-degenerate}.  With these, our result is:

\begin{theorem}[Structural stability of Arnol'd stable states]\label{2dEthmarnold}
  \label{existence}
 Let $\alpha\in(0,1)$ and $k \geq 3$. Let  $(M_0,g_0)$ be a compact two-dimensional Riemannian submanifold of $\mathbb{R}^2$ with $C^{k,\alpha}$ boundary.
  Suppose $\psi_0\in  C^{k,\alpha}(M_0)$ is a non-degenerate, Arnol'd stable steady state on $(M_0,g_0)$ with vorticity profile $F_0\in C^{k-2,\alpha}(\mathbb{R})$.
Then there are constants $\ve_1,\ve_2, \ve_3$ depending only on $ M_0, F_0, g_0$ and
$\|\psi_0\|_{C^{k,\alpha}}$ such that if $(M, g)$ is a compact Riemannian manifold
and $\rho:M_0\to \mathbb{R}$  with $\int_{D_0} \rho\  \rmd {\rm vol}_{g_0} = {\rm Vol}_g(D)$ and $g:M_0\to \mathbb{R}^2$ satisfy
\begin{align}
\|\partial M - \partial M_0\|_{C^{k,\alpha}}&\leq \ve_1,
\\
  \|\rho-1\|_{C^{k,\alpha}(M_0)}&\leq \ve_2,\\
  \|g-g_0\|_{C^{k,\alpha}(M_0)}&\leq \ve_3,
  \end{align}
there is a diffeomorphism $\gamma: M_0 \to M$ with Jacobian ${\rm det} (\nabla \gamma)= \rho$, and a function $F: \R \to \R$ close in $C^{k-2,\alpha}$ to $F_0$
so that $\psi = \psi_0\circ \gamma^{-1} \in C^{k,\alpha}(M)$ and $\psi$ satisfies \eqref{Euler1}--\eqref{Euler2} on $(M,g)$.
Thus, $u =\nabla^\perp_g \psi$  is a non-degenerate, Arnol'd stable steady
Euler solution on $(M,g)$ nearby $u_0$ whose vorticity $\omega =  F(\psi)$.
\end{theorem}

We remark that hypothesis necessary to run Theorem \ref{2dEthmarnold}, Hypothesis \textbf{(H1)}, is weaker than  Arnol'd stability since it allows the deformation of states of constant vorticity $F'(\psi_0)=0$.
Theorem \ref{2dEthmarnold} shows that such steady states are non-isolated from other stable stationary states, even fixing the domain and metric, since the Jacobian $\rho$ can be freely chosen.
 A very interesting open issue is whether any time-independent solution of the two-dimensional Euler equation can be isolated from other steady solutions.
Note that that the deformation scheme can be repeated to deformation between two ``far apart" domains provided along the path of steady states Hypothesis \textbf{(H1)} and  \textbf{(H2)} remain true.  {Finally, as discussed above, if one is not interested in controlling aspects of the streamline geometry of the new steady states, then an implicit function argument can be used to dispense with the non-degeneracy hypothesis \textbf{(H2)} and allow the construction of solutions with nearby vorticity profiles $F$. See, for example, the work of Choffrut and Sver\'{a}k  \cite{CS12}.}

\vspace{2mm}
\begin{center}
\textbf{Two-dimensional Boussinesq equations}
\end{center}
Given a domain $D_0\subset \mathbb{R}^2$ with smooth boundary, steady states of the Boussinesq system satisfy
\begin{align}\label{bous1}
u_0 \cdot \nabla u_0 &= -\nabla p_0+  \theta_0 e_2, \quad \text{ in } D_0, \\ \label{bous2}
  u_0\cdot \nabla \theta_0 &= 0, \quad\quad\quad\quad\quad\quad \   \text{ in } D_0,\\
\nabla \cdot u_0 &=0, \quad\quad\quad\quad\quad\quad \  \text{ in } D_0,\\
u_0\cdot \hat{n} &=0,    \quad\quad\quad\quad\quad\quad \ \text{ on }  \pa D_0.\label{bous3}
\end{align}
Introducing the vorticity $ \omega_0 = \nabla^\perp \cdot u_0$, equation \eqref{bous1} can be written as
\begin{align} \label{labelvortform}
  \omega_0 u^\perp_0  &=   \nabla P_0 + \theta_0  e_2, \\
   -P_0&:= p_0+ \frac{1}{2} |u_0|^2.
\end{align}
Letting $u_0=\nabla^\perp \psi_0$, $\omega_0= \Delta \psi_0 $ and $u^\perp_0=-\nabla \psi_0$ equation \eqref{labelvortform} reads
\begin{equation}
  -\Delta \psi_0 \nabla \psi_0   =\nabla P_0+ \theta_0 e_2.
 \label{}
\end{equation}
The Grad-Shafranov-like equation (analogous to equations \eqref{gseuler1}-
\eqref{gseuler2} of 2d Euler) is obtained by assuming
that $\theta_0, P_0$ can be constructed from the stream function, in
the sense that
\begin{align}
 \theta_0(y_1,y_2) &= \Theta_0(\psi_0(y_1,y_2)),\label{boustheta}\\
 P_0(y_1,y_2) &= -y_2 \Theta_0(\psi_0(y_1,y_2))  - G_0(\psi_0(y_1,y_2)),
 \label{bouspressure}
\end{align}
for smooth functions $G_0, \Theta_0$.  This again turns out to be completely general  provided that $u_0$ never vanishes, see Lemma \ref{eoslemma}.
In this case, provided that $G_0, \Theta_0$ are sufficiently smooth the stream function must satisfy
\begin{alignat}{2} \label{gsbous1}\
 \Delta \psi_0 - y_2 \Theta_0'(\psi_0) - G_0'(\psi_0) &= 0,  \quad\quad\quad\ \ \ \text{ in } D_0, \\
\psi_0 &= {\rm (const.)},  \quad \text{ on }  \pa D_0. \label{gsbous2}
\end{alignat}
Given a solution $\psi_0$ to \eqref{gsbous1}--\eqref{gsbous2}, the function $u_0 = \nabla^\perp \psi_0$ solves \eqref{bous1}-
\eqref{bous3} with temperature $\theta_0$ determined by \eqref{boustheta}
and pressure $P_0$ determined by \eqref{bouspressure}.

As for 2d Euler, if $D_0$ is a periodic channel the Boussinesq equations have a certain rigidity.  Specifically, \textit{all} smooth steady states with nowhere vanishing velocity must be shear flows and the temperature and pressures must satisfy the equations of state \eqref{boustheta},\eqref{bouspressure} for some  Lipschitz  functions $\Theta_0$ and $G_0$:

\begin{theorem} [Rigidity of non-stagnant Boussinesq flows] \label{liouBous}
Let $D_0$ be a periodic channel given by \eqref{channel} and suppose that $u_0: D_0\to \mathbb{R}^2$ and $\theta_0: D_0\to \mathbb{R}$ be a $C^2(D_0)$ solution of \eqref{bous1}--\eqref{bous3} with $\inf_{D_0} u_0> 0$.  Then there exists Lipschitz  functions $\Theta_0,G_0: \mathbb{R}\to \mathbb{R}$ such that  \eqref{boustheta},\eqref{bouspressure} hold and if furthermore
\be
\Theta_0'(\psi_0)\leq 0,
\ee
then $u_0$ is a shear flow, namely $u_0(x,y)= (v(y),0)$ for some scalar function $v(y)$.
\end{theorem}

 We next establish the flexibility of Boussinesq solutions by proving the existence of steady solutions on the channel to solutions on nearby domains. Fix $D_0$ to be a channel defined by \eqref{channel} and fix functions $\psi_0, \Theta_0, G_0$ satisfying
 \eqref{gsbous1}--\eqref{gsbous2} on $D_0$.
Given a  function $\Theta$, we then deform $\psi_0$ to obtain a steady solution $ \psi$ to   \eqref{gsbous1}   defined on $D$ given by \eqref{domain1} (for suitably small $b_0, b_1$) with temperature profile $\Theta$ and some vorticity profile $G$.  As a result, $u=\nabla^\perp \psi$ and $\theta = \Theta(\psi)$ satisfy the steady Boussinesq equations on $D$:
 \begin{align}\label{bous1f}
u \cdot \nabla u &= -\nabla p+  \theta e_2, \quad \ \ \  \text{ in } D, \\ \label{bous2f}
  u\cdot \nabla \theta &= 0, \quad\quad\quad\quad\quad\quad \   \text{ in } D,\\
\nabla \cdot u &=0, \quad\quad\quad\quad\quad\quad \  \text{ in } D,\\
u\cdot \hat{n} &=0,    \quad\quad\quad\quad\quad\quad \ \text{ on }  \pa D.\label{bous3f}
\end{align}
Specifically, we prove
\begin{theorem} [Flexibility of non-stagnant Boussinesq flows]  \label{2dBthm}
Let $D_0$ be defined by \eqref{channel} and  $D$ be defined by \eqref{domain1}.
  Suppose  $\psi_0:[0,1]\to \R$ is a shear with   $\psi_0 \in C^{k,\alpha}(D_0)$ for some $\alpha > 0, k \geq 3$ satisfying \eqref{gsbous1}--\eqref{gsbous2} for a given $G_0, \Theta_0\in C^{k-1,\alpha}(\mathbb{R})$ having no stagnation points $U_0:= \inf_{D_0} |\nabla \psi_0|>0$.
Then there are constants $\ve_1,\ve_2,\ve_3$ depending only on $U_0, D_0, \Theta_0, G_0$ and
$\|\psi_0\|_{C^{k,\alpha}}$ such that if $b_0, b_1:\R \to \R$, $\Theta :\R\to \mathbb{R}$ and $\rho:D_0\to \mathbb{R}$ with $\int_{D_0} \rho = {\rm Vol}(D)$ satisfy
\begin{align}
\|b_0\|_{C^{k,\alpha}(\R)}
  + \|b_1\|_{C^{k,\alpha}(\R)} &\leq \ve_1,
\\
  \|1-\rho\|_{C^{k,\alpha}(D_0)}&\leq \ve_2,\\
    \|\Theta_0'-\Theta'\|_{C^{k-2,\alpha}(\R)}&\leq \ve_3,
  \end{align}
there is a diffeomorphism $\gamma: D_0 \to D$  with Jacobian ${\rm det} (\nabla \gamma)= \rho$, and a function $G: \R \to \R$ close to $G_0$
so that $\psi = \psi_0\circ \gamma^{-1} \in C^{k,\alpha}(D)$ and $\psi$ satisfies
\be
\Delta \psi - x_2 \Theta'(\psi) - G'(\psi)=0\quad \text{ in } \quad D.
\ee
Thus, $u =\nabla^\perp \psi$ and $\theta = \Theta(\psi)$  defines a Boussinesq solution in $D$ nearby $u_0=\nabla^\perp \psi_0$,  $\theta_0=\Theta_0(\psi_0)$.
\end{theorem}

Note that, in light of Theorem \ref{liouBous}, if the base state $\psi_0$ has no stagnation points and $\Theta_0'\leq 0$, all smooth steady states are shears and so the assumption that
$\psi_0$ is a shear is automatic.

\vspace{2mm}
\begin{center}
\textbf{Three-dimensional axisymmetric Euler
}
\end{center}

Let $T_0\subset \R^3$ be a domain with smooth boundary. The three-dimensional steady Euler equations
(or, equivalently, the three-dimensional steady Magnetohydrostatic equations)
 read
\begin{alignat}{2}
\omega_0 \times u_0 &= \nabla P_0, &&\quad \text{ in } T_0, \label{steadymhd}\\
 \nabla \cdot u_0 &= 0, &&\quad \text{ in } T_0,\label{divB0}\\
 u_0\cdot \hat{n} &= 0, &&\quad \text{ on } \partial T_0,\label{Bbdy}
\end{alignat}
where $\omega_0= \nabla\times u_0$ denotes the vorticity and $P_0$ denotes
the pressure.

The issue of existence of solutions to \eqref{steadymhd}-\eqref{Bbdy} is of fundamental importance to the problem of magnetic confinement fusion. In particular, one strategy to achieve fusion is to drive a plasma contained in an axisymmetric toroidal domain (tokamak) towards an equilibrium configuration which is (ideally) stable and enjoys certain properties that make it suitable for confining particles which, to first approximation, travel along its magnetic field lines well inside the domain.  Once such a suitable steady state is identified, the control of the plasma to remain near this state is a very important and challenging engineering problem.  However, as Grad remarked in \cite{G67},
``\textit{Almost all stability analyses are predicated on the existence of an equilibrium
state that is then subject to perturbation. But a more primitive reason than instability
for lack of confinement is the absence of an appropriate equilibrium state." }
Grad goes on to write that there are exactly four known symmetries for which smooth toroidal
plasma equilibria with nested magnetic surfaces can exist. These are: two-dimensional,
axial, helical and reflection symmetries.
He asserts in \cite{G85} that ``\textit{no additional exceptions have arisen since 1967, when it
was conjectured that toroidal existence...of smooth solutions with simple nested surfaces
admits only these $\dots$ exceptions. $\dots$  The proper formulation of the nonexistence
statement is that, other than stated symmetric exceptions, there are no families
of solutions depending smoothly on a parameter.}" \
We formalize this statement as a rigidity property of solutions of three-dimensional Euler (Magnetohydrostatics):

\begin{conj}[H. Grad, \cite{G67,G85}]\label{GradsConj} Any non-isolated and non-vanishing (away from the ``magnetic axis") smooth unforced MHS equilibrium on a (topologically toroidal or cylindrical) domain $T\subset \mathbb{R}^3$ that has a pressure possessing nested level sets
which foliate $T$ has either plane-reflection, axial or helical symmetry.
\end{conj}

By an \textit{isolated stationary state}, we mean that, in some suitably topology, there are no nearby steady states aside from those which correspond to a trivial rescaling or translation of the original.  It is possible that no such object exists. The qualifier is included to make precise Grad's assertion that the conjecture apply to solutions which appear in  continuous``families".

Complementary to Grad's conjecture, here we prove that solutions with symmetry can also be severely restricted to conform to a stronger form of symmetry. Specifically, we consider periodic-in-$z$ solutions in the (hollowed out) axisymmetric cylinder (see right half of Figure \ref{fig:torus})
   \begin{align}   \label{domaintor}
T_0 = D_0 \times \mathbb{T}, \qquad &D_0 = \{ (r,z)\in \ [1/2,1]\times \mathbb{T}\},
\end{align}
which are axisymmetric in the sense that
$u_0 = u_0(r, z)$.
We
remark that solutions with this symmetry on this domain are not suitable for the confinement
of a plasma in a tokamak and instead describe steady flow in a pipe.
To find solutions with this symmetry, we make the ansatz
\begin{equation}
 u_0 = \frac{1}{r} e_\theta \times \nabla \psi_0 + \frac{1}{r} C_0(\psi_0) e_\theta
 \label{ansatz}
\end{equation}
for a function $C_0:\mathbb{R}\to \mathbb{R}$ and $\psi_0 = \psi_0(r, z)$ is to be determined.
In fact, by the results in \cite{BKM19}, any sufficiently smooth solution to
\eqref{steadymhd}-\eqref{Bbdy} possessing symmetry in the $\theta$ direction
and $\curl( u_0 \times e_\theta) = 0$ and possessing a nowhere vanishing pressure gradient is necessarily of the form \eqref{ansatz}.
If we seek a solution with pressure of the form $P_0 = \Pi_0(\psi_0)$ for some profile function $\Pi_0: \mathbb{R}\to \mathbb{R}$, then \eqref{ansatz} satisfies
\eqref{steadymhd}-\eqref{Bbdy} provided $\psi_0$ satisfies
\begin{alignat}{2}
  \frac{\pa^2}{\pa r^2} \psi_0 + \frac{\pa^2}{\pa z^2} \psi_0 - \frac{1}{r}
  \frac{\pa}{\pa r}\psi_0 &= -r^2\Pi_0'(\psi_0) +C_0 C_0'(\psi_0),
  \qquad && \text{ in } D_0,
 \label{gradshafranov}\\
 \psi_0 &= {\rm (const.)},  \quad && \text{ on }  \pa D_0. \label{gradshafranovbc}
\end{alignat}
The equation \eqref{gradshafranov} is known in plasma physics as the Grad--Shafranov equation \cite{GR58,S66}.\footnote{In fact, \eqref{gradshafranov} has been derived long before by Hicks in 1898 \cite{H98}. Consequently, in the fluid dynamics community, the same equation is known as the Hicks equation and also as the Bragg--Hawthorne equation  \cite{BH50} and the Squire--Long equation  \cite{L53,S56} due to independent re-derivations.
One can derive versions of \eqref{gradshafranov} for other symmetries as well;
see \cite{BKM19} for a generalization of \eqref{gradshafranov} in this direction.}

\begin{figure}[h!]
  \includegraphics[height=1.7in]{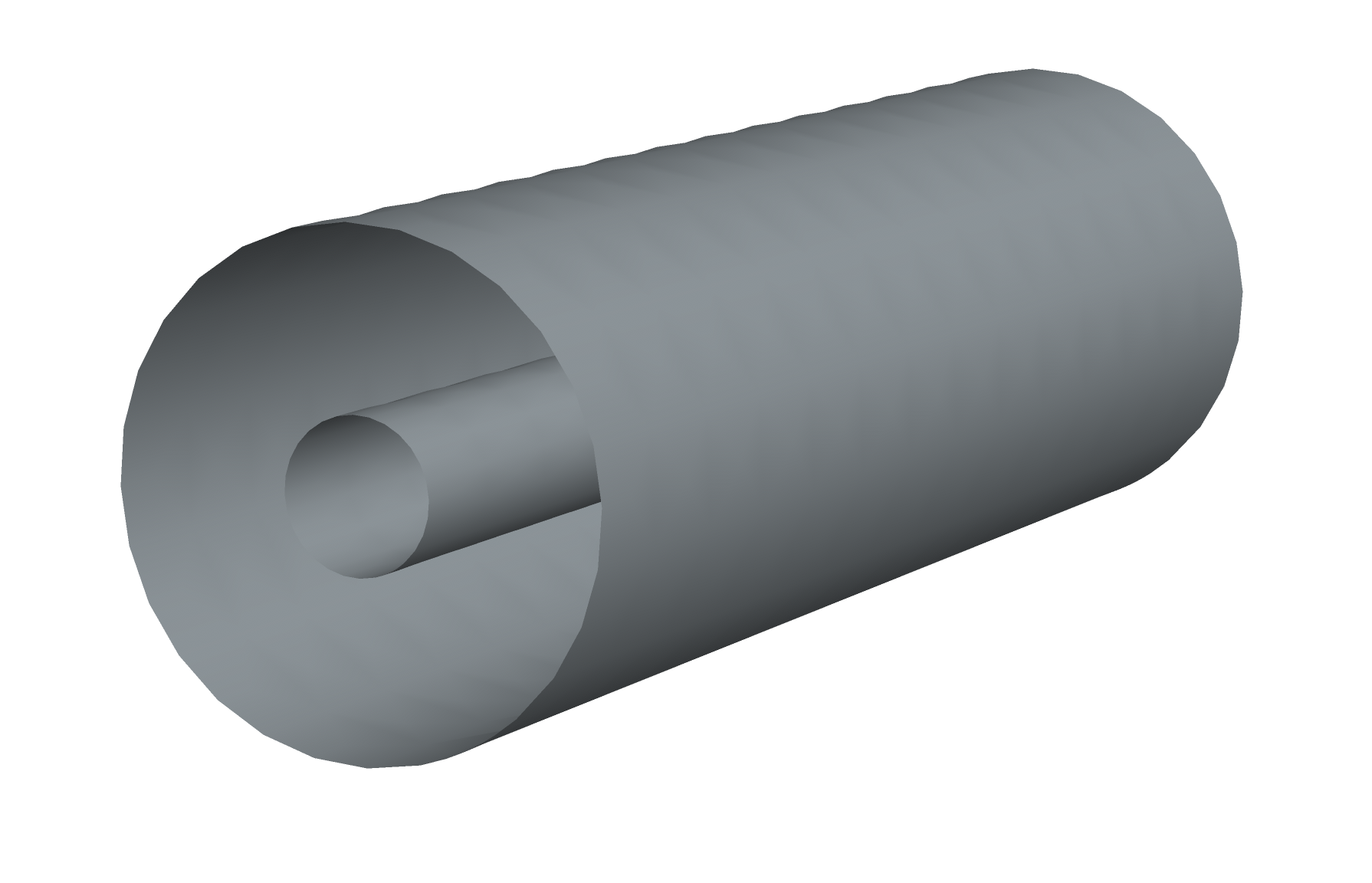}
  \includegraphics[height=1.7in]{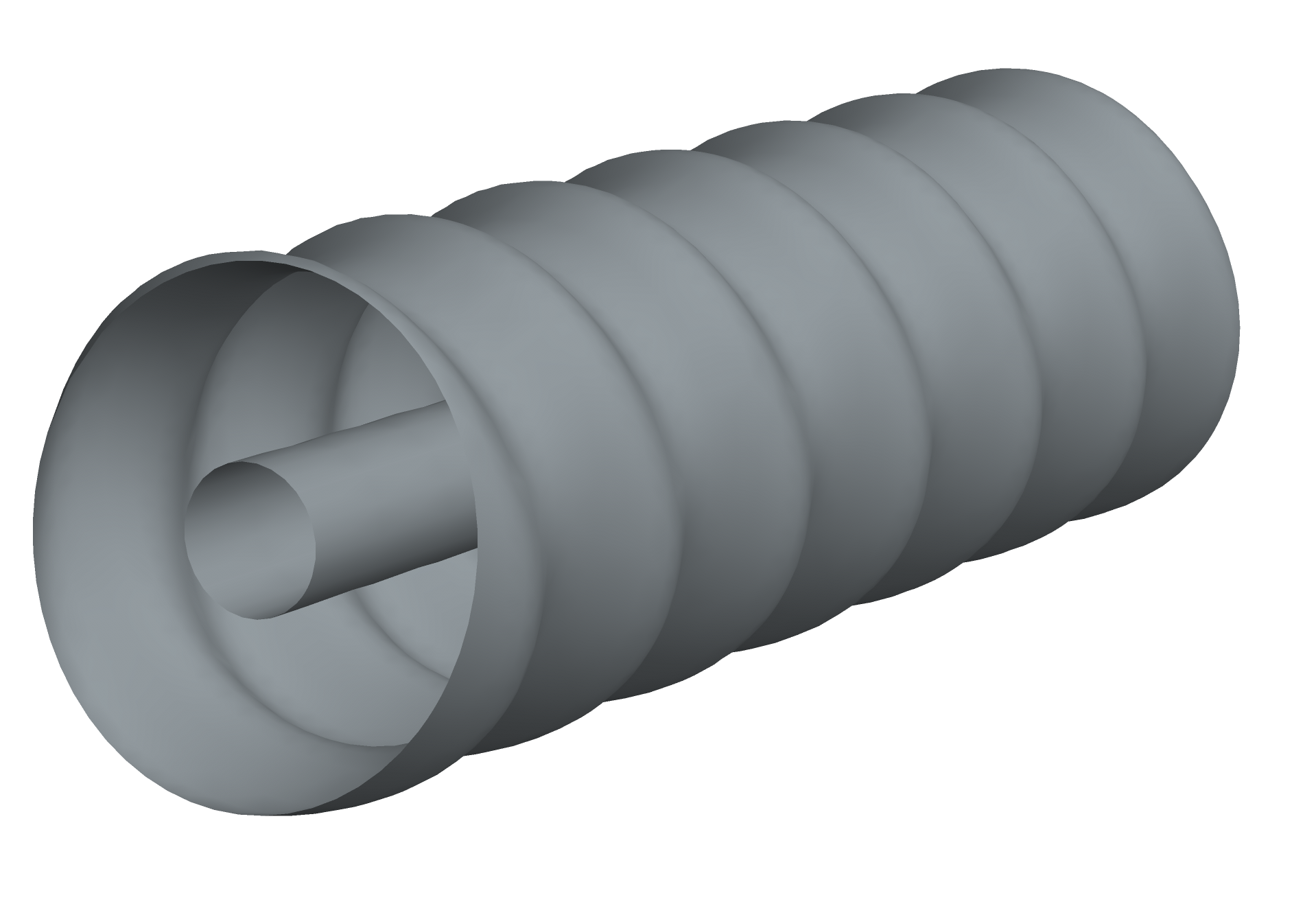}
  \caption{Axisymmetric deformation of the cylinder (corrugated pipe).}
  \label{fig:torus}
\end{figure}

Here we prove that all solutions whose pressure has a certain property must be radial.
\begin{theorem} [Rigidity of axisymmetric pipe flows] \label{liouMHD}
Let $D_0$ be given by \eqref{domaintor}. Suppose $\Pi_0,C_0: \mathbb{R}\to \mathbb{R}$ are Lipchitz functions and that $\psi_0: D_0\to \mathbb{R}$ is $C^2(D_0)$ solution of the Grad--Shafranov equation
 \eqref{gradshafranov}--\eqref{gradshafranovbc} with $\inf_{D_0} |\partial_r\psi_0|> 0$.  If furthermore
 $ \Pi_0$ satisfies
\be
\Pi_0'(\psi_0)\geq 0,
\label{condition1}
\ee
then $\psi_0$ is radial, i.e. $\psi_0(r,z)= \psi_0(r)$.
\end{theorem}
Physically, Theorem \ref{liouMHD} says that in order to support some non-trivial structure in pipe flow, the pressure cannot satisfy \eqref{condition1}.  It is conceivable that this has some bearing for identifying good flow configurations from the point of view of drag reduction.

 Liouville theorems constraining axisymmetric solutions of three-dimensional fluid equations have appeared previously in the work of Shvydkoy for Euler on $\mathbb{R}^3$ (\S 5 of \cite{RS18}) and by Koch-Nadirashvili-Seregin-Sver\'{a}k  \cite{KNSS} for ancient solutions of Navier-Stokes on $\mathbb{R}^3$.  Establishing similar rigidity results for the full three-dimensional problem outside of symmetry -- which is necessary to address Grad's conjecture -- seems to be out of reach of existing techniques.

We prove also the complementary flexibility result for periodic-in-$z$ solutions.  Specifically, we construct solutions of
\begin{alignat}{2}
  \frac{\pa^2}{\pa r^2} \psi + \frac{\pa^2}{\pa z^2} \psi - \frac{1}{r}
  \frac{\pa}{\pa r}\psi &= -r^2\Pi'(\psi) +C C'(\psi),
  \qquad && \text{ in } D
 \label{gradshafranov2}\\
 \psi &= {\rm (const.)},  \quad && \text{ on }  \pa D. \label{gradshafranovbc2}
\end{alignat}
where the domain occupied by the fluid is given by
   \begin{align}   \label{newcyl}
T = D \times \mathbb{T}, \qquad & D = \{ (r,z)\in \ [b_0(z), b_1(z)]\times \mathbb{T}\}.
\end{align}
  See the left half of Figure \ref{fig:torus} for a depiction of a possible domain.

\begin{theorem}[Flexibility of axisymmetric pipe flows]
  \label{3dMHDthm}
 Let $D_0$ be given by \eqref{domaintor} and  $D$ defined by \eqref{newcyl}.
 Fix $k \geq 3, \alpha \in (0,1)$.
   Suppose $\psi_0 \in C^{k,\alpha}(D_0)$ is a solution to the axisymmetric
   Grad-Shafranov equation \eqref{gradshafranov}-\eqref{gradshafranovbc}
   for some $\Pi_0, C_0 \in C^{k-1,\alpha}(\R)$,
   having no stagnation points in the sense that
    $U_0 := \inf|\nabla \psi_0 | > 0$
     in $D_0$. Suppose that additionally $\Pi_0, C_0$ and $\psi_0$ satisfy
    \be
\Pi_0'(\psi_0) >  \tfrac{1}{r^2} (C_0C_0')'(\psi_0).
  \ee
Then there are constants
  $\ve_1, \ve_2, \ve_3$ depending only on $U_0$ and
  $\|\psi_0\|_{C^{k,\alpha}(D_0)}, \|\Pi_0\|_{C^{k-1,\alpha}},$ and $\|C_0\|_{C^{k-1,\alpha}}$ such that if $b_0, b_1:\R \to \R$,  $\Pi:\R \to \R$  and $\rho:D_0\to \mathbb{R}$  with $\int_{D_0} \rho = {\rm Vol}(D)$ satisfy
\begin{align}
\|b_0\|_{C^{k,\alpha}(\R)}
  + \|b_1\|_{C^{k,\alpha}(\R)} &\leq \ve_1,
\\
  \|1-\rho\|_{C^{k,\alpha}(D_0)}&\leq \ve_2,\\
    \|\Pi_0'-\Pi'\|_{C^{k-2,\alpha}(\R)}&\leq \ve_3,
  \end{align}
  there is a diffeomorphism $\gamma: D_0 \to D$ and a function
  $C \in C^{k-1,\alpha}(\R)$ so that $\psi = \psi_0\circ \gamma^{-1}$ satisfies \eqref{gradshafranov2}
  in $D$ with pressure $\Pi$ and swirl $C$.
  In particular,
\begin{equation}
u =  \frac{1}{r}e_\theta \times \nabla \psi +  \frac{1}{r} C(\psi) e_\theta
\end{equation}
satisfies the Euler equation \eqref{steadymhd}--\eqref{Bbdy} with pressure $P = \Pi(\psi)$
in the domain $T= D\times \mathbb{T}$.
\end{theorem}
We remark, for the deforming scheme, that it is not necessary to cut out the inner part of the cylinder.
Theorem \ref{3dMHDthm} applies provided that \textbf{(H2)} on non-degeneracy is satisfied by $\psi_0$.

\section{Rigidity: Liouville Theorems}

To establish the claimed Liouville theorems, we first show that functions which satisfy steady transport by a velocity $u_0= \nabla^\perp\psi_0$ \textit{with no stagnation points} can be constructed from the streamfunction $\psi_0$ via a `nice' equation of state. This is the content of  Lemma \ref{eoslemma} below.

\begin{lemma}\label{eoslemma}
Fix $k\geq 3$ and let $D_0$ be diffeomorphic to the annulus and $\psi_0: D_0\to \mathbb{R}$ satisfy
\begin{itemize}
\item $\psi_0\in C^k(D_0)$,
\item  $\psi_0|_{ \pa D_0^{\rm bot}} = c_0 $ and $\psi_0|_{ \pa D_0^{\rm top}} = c_1 $  for constants $c_0\neq c_1$.
\item $|\nabla \psi_0|\neq 0$ in $D_0$.
\end{itemize}
Suppose that $\theta \in C^{k-2}(D_0)$ satisfies
\be
\nabla^\perp \psi_0 \cdot \nabla \theta = 0, \quad \text{in} \quad D_0.
\ee
Then, there exists a $(k-2)$--times continuously differentiable function $\Theta:\mathbb{R}\to \mathbb{R}$ such that
\be
\theta(x,y)= \Theta(\psi_0(x,y)),  \quad \text{in} \quad D_0.
\ee
\end{lemma}
\begin{proof}[Proof of Lemma \ref{eoslemma}]
Our   Lemma \ref{eoslemma} essentially appears as Lemma 2.4 of \cite{HN19} in the case when $D_0$ is the channel. We summarize
  the argument here for the sake of completeness. Given $p \in D_0$ and let $\xi_p = \xi_p(t)$ denote the
integral curve of $\nabla^\perp \psi$ starting at $p$ at ``time" $t=0$, namely
  \begin{equation}
   \frac{\rmd}{\rmd t} \xi_p(t) = \nabla^{\perp}\psi_0(\xi_p(t)), \qquad
   \xi_p(0) = p, \qquad t\in \mathbb{R}.
   \label{}
  \end{equation}
  By Lemma 2.2 of \cite{HN19}, $\xi_p(t)$ is uniquely defined for all $t\in \mathbb{R}$ and is periodic in $t$ and moreover
  the curve $\xi_p(\R)$ passes through each $x\in [0,2\pi)$. Identifying the periodic channel with the annulus, this means that the curve  $\xi_p(\R)$ surrounds the inner disc.
  Given $q \in D_0$ we also
  let $\sigma_q$ denote the integral curve of $\nabla \psi$,
  \begin{equation}
   \frac{\rmd}{\rmd t} \sigma_q(t) = \nabla\psi_0(\sigma_q(t)), \qquad
   \sigma_q(0) = q, \qquad t\in \mathbb{R}.
   \label{}
  \end{equation}
  We now fix any point $q = (q_1, 0)$ at the bottom of the channel $\{y=0\}$.
  As a consequence of the fact that the vector field $\nabla\psi_0$ points normal to the boundary, it is shown in \cite{HN19} that there is a $t_q < \infty$ so that
  $\sigma_q(t)$ lies at the top of the channel, $\sigma_q(t_q) = (q_2, 1)$.
  Writing $g(t) = \psi(\sigma_q(t))$, we have $g'(t) = |\nabla\psi_0(\sigma_q(t))|^2
  > 0$ so it follows that $g$ is invertible with
  $C^{k-2}$ inverse. We define $\Theta$ by
  \begin{equation}
   \Theta(\tau) = \theta(\sigma_q(g^{-1}(\tau))).
   \label{}
  \end{equation}
  Then $\Theta$ is $C^{k-2}$ and $\Theta(\psi_0(\sigma_q(s))) = \theta(\sigma_q(s))$ for any $s$.  Finally, fix now  any point $p\in D_0$. For large enough $t$, there is an $s$ so that $\xi_p(t) =
  \sigma_q(s)$. Since $\nabla^\perp\psi_0\cdot \nabla \theta = 0$, we have
  $\theta(p) = \theta(\xi_p(t)) = \theta(\sigma_q(s))$. This completes the proof since then we have
  \begin{equation}
   \theta(p) = \theta(\sigma_q(s)) = \Theta(\psi_0(\sigma_q(s))) = \Theta(\psi_0(p)).
   \label{}
  \end{equation}
\end{proof}
We will need the following result which ensures that the stream function takes different values at the top and bottom, and ranges between these values in the interior:
\begin{lemma} [Lemma 2.6 of \cite{HN17}, Lemma 2.1 of \cite{HN19}] \label{eoslemma2}
Let $D_0$ be diffeomorphic to the annulus  and let  $\psi_0\in C^3(D_0)$  with  $|\nabla \psi_0|\neq 0$ on $D_0$ satisfy
\be
\psi_0|_{ \pa D_0^{\rm bot}} = c_0, \qquad\psi_0|_{ \pa D_0^{\rm top}} = c_1,
\ee
for some $c_0, c_1\in \mathbb{R}$.  Then $c_0\neq c_1$ and
\be
\min\{c_0,c_1\}< \psi_0  <\max\{c_0,c_1\}, \quad \text{on} \quad D_0\setminus \partial D_0.
\ee
\end{lemma}
The proof of this result can be found in the cited references. There, \ref{eoslemma2} is established when $D_0$ is periodic channel  \eqref{channel}, but an inspection of the proof shows that the result holds more generally.

Finally, we prove the corresponding Liouville theorem which modestly generalizes
Theorem 1.6 of \cite{HN17} to accommodate the additional terms arising in the settings of the Boussinesq and axisymmetric Euler equations.
\begin{theorem}[Liouville Theorem]
  \label{hn}
 Let ${D_0} = \mathbb{T}\times[1/2,1]$ and let
  $f = f(y), g = g(y,\psi), h = h(\psi)$ be Lipschitz functions.
  Let $\psi\in C^2(D_0)$ be a solution to
 \begin{equation}
  \Delta \psi + f(y) \pa_y \psi + g(y, \psi) + h(\psi) = 0, \quad \text{ in } D_0,
  \label{}
 \end{equation}
where $\psi$ is periodic in $x\in \mathbb{T}$ with boundary conditions
 \begin{align}
  \psi(x, 1/2) &= 0,\qquad
  \psi(x, 1) = c > 0.
  \label{}
 \end{align}
Suppose that one of the following conditions holds
\begin{itemize}
\item  $g_y, f_y \geq 0$, and $0< \psi < c$ in ${D_0},$
\item   $g_y, f_y \geq 0$, and  $\psi_y \geq 0$ on ${D_0}.$
\end{itemize}
Then
 $\psi$ is independent of $x$, namely $\psi := {\psi}(y)$.
\end{theorem}
\begin{proof}[Proof of Theorem \ref{hn}]
 The proof is nearly identical to the one in \cite{HN17} with minor extension to accommodate $f$ and $g$.  For the sake of completeness we include a proof here.
Fix $\xi \in \R^2$ with $\xi = (\xi_1, \xi_2)$ with $\xi_2 > 0$.
For $\tau \in (0, 1/\xi_2)$,  set
\begin{equation}
 {D_0}^\tau =\mathbb{T} \times (1/2, 1-\tau \xi_2),
 \label{}
\end{equation}
and
\begin{equation}
 w^\tau(x) = \psi(x + \tau \xi) - \psi(x), \quad
 x \in \overline{{D_0}}^\tau.
 \label{wtaudef}
\end{equation}
Then the main ingredient for the proof of Theorem \ref{hn} is the following lemma
\begin{lemma}
  \label{wtaulem}
 For $w^\tau$ defined by \eqref{wtaudef} we have
 \begin{equation}
  w^\tau > 0 \quad \text{ in } \overline{{D_0}}^\tau, \quad
  \text{ for all } \tau \in (1/2, 1/\xi_2).
  \label{}
 \end{equation}
\end{lemma}
We first prove Theorem \ref{hn} assuming the result of this lemma.
  \label{wtaupf}
Note that $\psi \geq 0$ on ${D_0}$, since $\psi_y \geq 0$ on ${D_0}$ by assumption and the boundary values are $\psi|_{y=1/2}=0$ and $\psi|_{y=1}=c>0$.
  Taking $\xi_2 \to 0$ in the inequality $w^\tau > 0$ shows that
  \begin{equation}
   \psi(x + \tau \xi_1, y) \geq \psi(x, y).
   \label{decreasingshift}
  \end{equation}
  This holds for any $\xi_1 \in\R$ and we claim that this implies that
  we actually have equality in \eqref{decreasingshift}. Indeed, suppose
  that there are $x, \tau, \xi_1, y$ so that
  $\psi(x + \tau\xi_1, y) > \psi(x, y)$. Applying \eqref{decreasingshift} we have
  \begin{equation}
   \psi(x, y) = \psi(x -\tau\xi_1 + \tau \xi_1, y) \geq
   \psi(x + \tau \xi_1, y) > \psi(x, y),
   \label{}
  \end{equation}
  a contradiction. This completes the proof.
\end{proof}

\begin{proof}[Proof of Lemma \ref{wtaulem}]
Set
\begin{equation}
 \tau_* = \inf\{\tau \in (1/2, 1/\xi_2)\ \ \text{such that} \ \ w^{\tau'} > 0 \ \  \text{ in }
 \overline{{D_0}}_{\tau'} \text{ whenever } \tau' \in (\tau, 1/\xi_2)\}.
 \label{}
\end{equation}
By the maximum principle for narrow domains \cite{GNM79} we have $\tau_* < 1/\xi_2$.
We are going to prove that
$\tau_* = 1/2$. Suppose that instead $\tau_* > 1/2$. Then
$w^{\tau^*} \geq 0$ in $\overline{{D_0}}^{\tau_*}$ and there are sequences
$\tau^k \in (1/2, \tau_*]$ and $(x^k, y^k) \in {D_0}$ so that
\begin{equation}
 (x^k, y^k) \in \overline{{D_0}}^{\tau_k}, \quad
 \text{ and } \quad w^{\tau_k}(x^k, y^k) \leq 0.
 \label{}
\end{equation}
Define
\begin{equation}
 \psi_k(x,y) = \psi(x + x^k, y), \qquad \text{for} \qquad  (x,y) \in {D_0}.
 \label{}
\end{equation}
The functions $\psi_k$ are uniformly bounded in
 $C^{2,\alpha}({D_0}^{\tau_*})$,
and so we can extract a convergent subsequence with $\psi_k \to \Psi \in
C^{2}({D_0}^{\tau^*})$. Taking $k \to \infty$ we see that
 $
 0 \leq \Psi \leq c.$
 We now show that these inequalities are strict.
Taking $k\to \infty$
in the equation for $\psi$ and differentiating in $y$ we see that
\begin{equation}
 \Delta \Psi_y + f(y) \pa_y \Psi_y +
 \big(f_y(y) + g_{\Psi}(y, \Psi) +h_{\Psi}(\Psi)\big) \Psi_y
 = -g_y(y, \Psi)\leq 0,
 \label{}
\end{equation}
by assumption.
Since we also have $\Psi_y \geq 0$ on the boundary, it follows
from the maximum principle for non-negative functions
 that $\Psi_y > 0$ in the interior as well, and so
\begin{equation}
 0 < \Psi < c.
 \label{strict}
\end{equation}
Next, the points
$y^k$ are bounded and so we can extract a convergent subsequence
$y^k \to \tilde{y}$. We have
\begin{equation}
 \Psi(\tau_* \xi_2, \tilde{y} + \tau_* \xi_2) = \Psi(0, \tilde{y}),
 \label{contra}
\end{equation}
because $w^{\tau_*} \geq 0$ in ${D_0}^{\tau_*}$ and $w^{\tau_k} (x^k, y^k) \leq 0$.
If $(0, \tilde{y}) \in \pa {D_0}^{\tau_*}$
then either $\tilde{y} = 0$ or $\tilde{y} = 1- \tau_*\xi_2$. But
by \eqref{contra} and \eqref{strict} neither of these are possible.
The only possibility left is $(0, \tilde{y}) \in {D_0}^{\tau_*}$.
Set
\begin{equation}
 W(x) = \Psi(x+ \tau_* \xi) - \Psi(x),
 \label{}
\end{equation}
then writing $\Psi_{\tau_*}(x) = \Psi(x + \tau_*\xi)$
we see that in $D_0^{\tau_*}$, since $f, g$ are Lipschitz
in $\psi$ there is an $L^\infty$ function $c = c(x,y)$ so that
\begin{equation}
 \Delta W +f(y) \pa_y W + c(x,y)W =
 \big(f(y) - f(y+\tau_* \xi_2)\big)\pa_y \Psi_{\tau^*}
 +g(y, \Psi_{\tau^*}) - g(y+\xi_2\tau_*, \Psi_{\tau^*})
 \leq 0
 \label{} \nonumber
\end{equation}
because $\xi_2 \geq 0$ and that
$\pa_y \Psi \geq 0$ (which is only needed $f$ is nonzero) and that
$f,g$ are increasing in $y$. Also we have
$W \geq 0$ in ${D_0}$, $W \geq 0$ on $\pa {D_0}$. By the maximum
principle for non-negative functions this implies that $W \equiv 0$
and in particular $W = 0$ on $\pa {D_0}$.  As we have shown that this is
impossible, we conclude $\tau_* = \frac{1}{2}$.
\end{proof}

\subsection{Proof of Theorem \ref{liouE} }
We assume $\psi_0\in C^3(D_0)$.  Note that, since the vorticity  satisfies
\be
u_0\cdot \nabla \omega_0 =0
\ee
and $|u_0|\neq 0$ and $\omega_0\in C^1(D_0)$, Lemma \ref{eoslemma} implies that there exists a $C^1(\mathbb{R})$ function $F_0$ such that $\omega_0=F_0(\psi_0)$.  Consequently, the stream function $\psi_0$ satisfies the elliptic equation
\begin{align}
\Delta \psi_0 = F_0(\psi_0) & \quad \text{ in } D_0, \\
\psi|_{\partial D_0^{\rm top} } = c_1, &\quad \psi|_{\partial D_0^{\rm top} } = c_2,
\end{align}
for some constants $c_1$ and $c_2$ with $c_1\neq c_2$ by Lemma \ref{eoslemma2}. Without loss of generality, we may take $c_1=0$ and $c_2>0$ by shifting $\psi_0\mapsto \psi_0- c_1$, sending $\psi_0 \mapsto -\psi_0$ if $c_2<0$,
and replacing $F_0(\psi_0)$ with $\pm F_0(\pm \psi_0 + c_1)$. Moreover, by Lemma \ref{eoslemma2} we have
\be
0 < \psi_0 < c_2 \quad \text{ in } D_0.
\ee
Applying Theorem \ref{hn} with $b=0$, $f=0$ and $g=-F_0$ gives the result.

\subsection{Proof of Theorem \ref{liouBous} }

We argue as in the proof of Theorem \ref{liouE}, but we apply
Theorem \ref{hn} with $f=0$, $g(y,\psi)=y \Theta_0'(\psi)$ and $h(\psi)= - G_0(\psi)$.

\subsection{Proof of Theorem \ref{liouMHD} }
Assuming \eqref{condition1}, the proof follows as in Theorem \ref{liouE}, but now $f=-\frac{1}{r}$,
$g(y,\psi)=r^2 \Pi_0'(\psi)$ and $h(\psi)= C_0C_0'(\psi)$.

\section{Flexibility: Deforming Domains}\label{flexsec}

We prove here a more general theorem, which covers the specific settings of Theorems \ref{2dEthm},  \ref{2dBthm} and  \ref{3dMHDthm}. We now outline the general setup.
Consider two bounded domains $D_0, D\subset \mathbb{R}^2$ given by the zero level sets of functions $B_0, B: \mathbb{R}^2 \to \mathbb{R}$:
\begin{equation}
 \pa D_0 = \{ B_0 = 0\}, \qquad  \pa D = \{ B = 0\}.
 \label{bdry}
\end{equation}
It is convenient to denote points in $D_0$ by $y = (y_1,y_2)$ and points
 in $D$ by $x =(x_1,x_2)$.  We will consider the problem of solving a certain elliptic equation on $D$ by deforming a solution of a `nearby' elliptic equation on $D_0$. \\

 \noindent \textbf{Elliptic equation on $D_0$}:
Consider a second-order elliptic operator on $D_0$ of the form
\begin{equation}
 L_0 = \sum_{i,j = 1}^2a_0^{ij}(y)\frac{\pa}{\pa y^i}\frac{ \pa}{\pa y^j} +
 \sum_{i = 1}^2 b_0^i(y)\frac{\pa}{\pa y^i},
 \label{L0def}
\end{equation}
where $a_0^{ij}, b_0^i$ are smooth functions defined on $\R^2$ and where the matrix $a_0^{ij}$
satisfies
\begin{equation}
 a^{ij}_0z_i z_j \geq M |z|^2, \qquad \forall z\in \mathbb{R}^2
 \label{ellipticity}
\end{equation}
for some $M > 0$.
We assume that we have a solution $\psi_0$ to the following nonlinear equation
 \begin{align}
  L_0 \psi_0 &=   F_0(\psi_0) + G_0(y, \psi_0), \quad \text{ in } D_0,
  \label{modelpbm}\\
  \psi_0 &= {\rm (const.)}, \quad\quad\quad\quad\quad\ \ \text{ on } \partial D_0,
 \end{align}
with functions $F_0:\mathbb{R}\to \mathbb{R}$ and $G_0:D_0\times \mathbb{R}\to \mathbb{R}$.\\

  \noindent \textbf{Elliptic equation on $D$}:
Given
coefficients $a^{ij}, b^i$ defined on $\R^2$, we set
\begin{equation}
 L = \sum_{i,j = 1}^2a^{ij}(x)\frac{\pa}{\pa x^i}\frac{\pa}{\pa x^j} +
  \sum_{i = 1}^2 b^i(x)\frac{\pa}{\pa x^i},  \qquad
 \label{}
\end{equation}
which is assumed to be elliptic as in \eqref{ellipticity}.
Consider the following equation for $\psi$
 \begin{align}
  L\psi &=  F(\psi) + G(x ,\psi), \quad \text{ in } D,
  \label{modelpbmdef}\\
  \psi &= {\rm (const.)} , \quad\quad\quad\quad\  \text{ on } \partial D,
 \end{align}
with  functions $F:\mathbb{R}\to \mathbb{R}$ and $G:D\times \mathbb{R}\to \mathbb{R}$.
\\

\noindent
\textbf{Problem:}\textit{
Let  $D$ and $D_0$ by two nearby domains (in the sense that $B$ and $B_0$ are close)
Let $F_0$, $G_0$ and a solution $\psi_0$ to \eqref{modelpbm} on $D_0$ be given.  Let $G$ be a given function close to $G_0$.  Find a diffeomorphism $\gamma:D_0\to D$ and a function
$F$ close to $F_0$ so that the function
\begin{equation}
 \psi = \psi_0\circ \gamma^{-1}
 \label{psidefinition}
\end{equation}
is a solution to \eqref{modelpbmdef}}.
\\

\vspace{-2mm}

The important observation of \cite{WV05} is that if we write
$\gamma ={\rm id} + \nabla \eta + \nabla^\perp \phi$ for functions $\eta, \phi$,
then plugging \eqref{psidefinition} into \eqref{modelpbmdef} leads to
an Dirichlet problem for $\pa_s \phi := \nabla^\perp \psi_0\cdot \nabla \phi$.
The function $\eta$ is free in the problem but if one wants
to fix the value of the Jacobian determinant $\rho := \det \nabla \gamma$,
$\eta$ can be determined by solving a Neumann problem. We formalize this
in the following Proposition.
\begin{prop}[Elliptic system for diffeomorphism]\label{pset}
  Fix two domains $D_0, D \subseteq \R^2$
  as in \eqref{bdry} and a solution to \eqref{modelpbm} $\psi_0:D_0 \to \R$.
  Let $ F_0, G_0$ and $G$ be given. Let $\rho: D_0 \to \R$ be a given continuous
  function  such that $\int_{D_0} \rho = {\rm Vol}(D)$.
  Suppose that $\eta, \phi:D_0 \to \R$ satisfy
 \begin{align}\label{etaeqnprop}
\Delta \eta&= \rho-1 +   \sN_\eta,\\
  (L_0 - \Lambda) \partial_s \phi  &=(F-F_0)(\psi_0)   + \sL_\phi  + \sN_\phi , \label{phieqnprop}
\end{align}
 for some $F = F(\psi_0)$, where $L_0$ is as in \eqref{L0def},
 $\Lambda :=F_0'(\psi_0) + (\partial_{\psi}G_0)(y,\psi_0)$,
 where $\partial_s \phi = \nabla^\perp \psi_0\cdot \nabla \phi$, and
where
\be
 \sL_\phi=\sL_\phi(\delta a, \delta b,\delta F,  \partial\delta G, \partial a_0, \partial b_0, \partial^3 \eta, \partial\rho, \partial^2 \partial_s\phi, \partial_{\psi_0} \phi; \partial^3 \psi_0)
\ee is defined by \eqref{Lphi} consists of terms which are linear in $\phi$ and $\eta$ (and their derivatives), multiplied by small factors, where
\begin{align}
 \sN_\eta&= \sN_\eta(\partial^2 \eta, \partial^2\phi)\\
\sN_\phi &=\sN_\phi (\partial^2 a_0, \partial^2 b_0, \partial^2\phi, \partial^2 \eta,\partial\rho, \partial^2 \partial_s \phi; \partial^3 \psi_0)
\end{align}
are nonlinearities with
$\sN_\eta$ is defined by \eqref{Neta} and $\sN_\phi$  by  \eqref{Nphi},
and where $\delta a = a - a_0$ and similarly for $\delta b, \delta F, \delta G$.
If $\gamma = {\rm id} +  \nabla^\perp \phi + \nabla \eta$ is   a diffeomorphism  $\gamma:D_0\to D$ with
$\det \nabla \gamma =\rho$,  then
 the function $\psi = \psi_0\circ \gamma^{-1}$ is a solution of \eqref{modelpbmdef} in $D$.
\end{prop}
This Proposition is proved in \S \ref{changeofcoords} (see Lemma \ref{philem}).
With this in hand, we address the above problem by constructing solutions with one (infinite dimensional) degree of freedom fixed by choosing the Jacobian of the map.  This requires three hypotheses on $\psi_0$ and the quantities in \eqref{modelpbmdef}.
\\

We need one hypothesis on the invertibility of the operator appearing in Proposition \eqref{pset} so the $\partial_s\psi$ can be recovered from  eqn. \eqref{phieqnprop} at the linear level.
We view $L_0:H_0^1\cap H^2 \to L^2$ and require:\\

\vspace{-2mm}
\noindent  {\rm Hypothesis 1 (\textbf{H1}): } Let $\Lambda =F_0'(\psi_0) + (\partial_{\psi}G_0)(y,\psi_0)$.  The problem
\begin{align}
  \left(L_0 - \Lambda\right) u &= 0 \qquad \text{in}\ \   D_0,\\
  u&=0  \qquad \text{on}\ \   \partial D_0,
\end{align}
admits only the trivial solution in $H^1_0(D_0)$.
\\

It is easy to see that Hypothesis  {\rm (\textbf{H1})} is guaranteed if $\Lambda$ avoids the discrete spectrum of $-L_0$, an open condition. In light of this, a stronger but easier to verify hypothesis that implies (\textbf{H1}) is
\\

\noindent  {\rm Hypothesis 1$^\prime$ (\textbf{H1}$^\prime$): } The operator $  \left(L_0 - \Lambda\right) $ is positive definite, i.e. for  all $f\in H^1_0(D_0)$ there is a constant $C>0$ such that
$
\langle   \left(L_0 - \Lambda\right) f, f\rangle_{L^2(D_0)} \geq C \|f \|_{H^1(D_0)}^2.$\\

This holds in the case of the 2d Euler equation if the base state is Arnol'd stable or if it is a shear flow without stagnation points (see Lemma \ref{kerlem}).

The next two hypotheses are needed in order to recover
$\phi$ from $\pa_s\phi$ once the latter is obtained by solving eqn. \eqref{phieqnprop} using (\textbf{H1}).  Since $\pa_s = \nabla^\perp \psi_0\cdot \nabla$,
in order to recover $\phi$, we must be able to integrate along
 streamlines of $\psi_0$ which requires a certain non-degeneracy of the base state.
 On a multiply connected domain diffeomorphic to the annulus, the base state must have no stagnation points (points at which $\nabla \psi_0=0$). On a simply connected domain diffeomorphic to a disc, there must be exactly one stagnation point.  This is quantified by the following hypothesis on the ``travel-time"  $\mu$ of  a parcel moving at speed $|\nabla \psi|$ to make a complete revolution on a streamline:
\\

\noindent  {\rm Hypothesis 2 (\textbf{H2}): } Let $I={\rm im}(\psi_0)$. There exists a constant $C>0$ so that
\be
 \mu(c) = \oint_{\{\psi_0=c\}} \frac{\rmd \ell}{|\nabla \psi_0|}
\leq C  \qquad \text{for all}\ \   c\in I
\ee
where $\ell$ is the arc-length parameter.  Note if $\psi_0\in C^{k,\alpha}(D_0)$ then $\mu \in C^{k-1,\alpha}(I)$.   \\

Finally, we  need an additional hypothesis that allows us to recover $\phi$ once we solve eqn. \eqref{phieqnprop}  for $\partial_s \phi$.
Specifically, at the linear level, $\phi$
needs to be chosen to satisfy $(L_0 - \Lambda)  \pa_s \phi = F + N,$
for a given function $N$ and for $F$ to be determined. An obvious necessary
condition for solvability is that $  (L_0 - \Lambda)_{\rm hbc}^{-1} (F + N)$ should have integral
zero along streamlines.
  We must therefore be able to choose $F$ to satisfy
 this condition while maintaining that $F = F(\psi_0)$ is a function
 \textit{only of the stream function} in order for the resulting $\psi$ to
 solve the correct equation. \\

\noindent  {\rm Hypothesis 3 (\textbf{H3}): } Fix $k\geq 2$, $\alpha\in (0,1)$ and let $I={\rm im}(\psi_0)$. Let $K_{\psi_0}: C^{k-2,\alpha}(I)\to C^{k,\alpha}(I)$ be
\be\nonumber
(K_{\psi_0} u)(c):= \frac{1}{\mu(c)} \oint_{\{\psi_0=c\}}  (L_0 - \Lambda)^{-1}_{\rm hbc}[u\circ \psi_0]\   \frac{\rmd \ell}{|\nabla \psi_0|} , \qquad c\in I.
\ee
For any $g\in C^{k,\alpha}(I)$ such that $g(\psi_0(\partial D_0))=0$,  there exists a $u\in C^{k-2,\alpha}(I)$ such that $K_{\psi_0} u = g$.  Moreover, $\|u\|_{C^{k-2,\alpha}(I)} \lesssim \|g\|_{C^{k,\alpha}(I)}$.
\\

It turns out that {\rm (\textbf{H3})} is a consequence of   {\rm (\textbf{H1}$^\prime$)} and {\rm (\textbf{H2})}.  To prepare for the proof, we define the streamline projector $\mathbb{P}_{\psi_0}$ which maps functions on $D_0$ to functions which are constant on level sets of the streamfunction $\psi_0$,
\be
(\mathbb{P}_{\psi_0} f)(c):=  \frac{1}{\mu(c)} \oint_{\{\psi_0=c\}}  f\  \rmd s, \qquad \text{for all}\ \   c\in I
\ee
where $\rmd s= \rmd \ell / |\nabla \psi_0|$.
This operation is well defined on functions which can be integrated on curves (e.g. functions that are in $H^1(D_0)$) by Hypothesis  {\rm (\textbf{H2})}. With this notation we have
$
K_{\psi_0} u:=  \mathbb{P}_{\psi_0}  (L_0 - \Lambda)^{-1}_{\rm hbc}[u].
$
Note that if $f,g$ are such that $\mathbb{P}_{\psi_0} f=0$ and $\mathbb{P}_{\psi_0} g=g$ then
\be\label{ortho}
\int_{D_0} f g = \int_I\left( \oint_{\{\psi_0=c\}} f g \rmd s\right) \rmd c =  \int_I  g \left(\oint_{\{\psi_0=c\}} f  \rmd s\right) \rmd c = 0.
\ee
Here we use the fact that $\psi_0$ satisfies  {\rm (\textbf{H2})} and therefore has streamlines which foliate $D_0$ so we can use action-angle coordinates to compute the integral \eqref{ortho}.
For further discussion see \S \ref{stream} herein or the textbook
\cite{A89}.
It follows that $\mathbb{P}_{\psi_0}$ is orthogonal in $L^2(D_0)$, i.e. for any $h\in C(D_0)$ we have
\begin{align}
\| h\|_{L^2}^2 &= \int_{D_0} \Big(| \mathbb{P}_{\psi_0}h|^2 + 2 (\mathbb{P}_{\psi_0}h) (\mathbb{Q}_{\psi_0} h) + |\mathbb{Q}_{\psi_0}h|^2 \Big) = \|\mathbb{P}_{\psi_0}h\|_{L^2}^2 + \|\mathbb{Q}_{\psi_0}h\|_{L^2}^2\label{PQform}
\end{align}
where $\mathbb{Q}_{\psi_0}= \mathbf{1} - \mathbb{P}_{\psi_0}$.   In light of these properties, $\mathbb{P}_{\psi_0}$ is a projection on $L^2$.

The
motivation for Lemma \ref{h1h2h3lem} in a Hilbert space $H$ is that if $P$ is a projection ($P^2=P$  and $P^*=P$)  and $A$ is  bounded positive operator then the compression
$PAP$ is positive in $PH$ since
$$
\langle PAPx,x\rangle_H = \langle APx,Px\rangle_H\ge C\langle Px,Px\rangle_H.
$$
The fact that $A$ is bounded is used only to make sure that $PH$ is included
in the domain of $A$.

\begin{lemma}\label{h1h2h3lem}
Fix $k\geq 2$ and suppose Hypotheses  {\rm (\textbf{H1}$^\prime$)} and {\rm (\textbf{H2})}  hold.
Then {\rm (\textbf{H3})} holds.
\end{lemma}

The proof is deferred to \S \ref{h1h2h3lemsec}.
We remark that, invertibility of $\L_0 -\Lambda$ alone (Hypothesis  {\rm (\textbf{H1})}) cannot be expected to imply Hypothesis  {\rm (\textbf{H3})} itself as is easily demonstrated in finite dimensions.  Positive definiteness is a crucial point in our argument.
We finally note that if we further know that  $(L_0 - \Lambda)g(\psi_0)$ is itself a function of $\psi_0$, which is the case when the base solution and the operator $L_0$ enjoy some mutual symmetry, we can find the solution of Hypothesis  {\rm (\textbf{H3})} explicitly

\begin{lemma}\label{h3lem}
Suppose for any $f\in C^{k,\alpha}(I)$, the function  $(L_0 - \Lambda)f(\psi_0)$
depends only on the value of the stream function,
$(L_0 - \Lambda)f(\psi_0) = h(\psi_0)$ for some $h\in C^{k-2,\alpha}(I)$ . Then  {\rm \textbf{(H3)}} holds with $u=(L_0 - \Lambda)g$.
\end{lemma}

Our main theorem on deforming solutions of elliptic equations is a quantitative  version of the implicit function theorem.  As stated above, this generalizes the setup and results of Wirosoetisno and Vanneste  \cite{WV05}.  It is also  similar in spirit to the result of Choffrut and Sver\'{a}k  \cite{CS12} which, on annular domains,  establishes  a one-to-one correspondence between vorticity distribution functions and steady states of two-dimensional Euler  nearby a solution satisfying a version of {\rm (\textbf{H1})} (see also \cite{D17}).  In our theorem, $F:= F(\psi)$ (which plays the role of the vorticity distribution function for 2d Euler) is not chosen ahead of time but rather  accommodates the deformation of the other parameters (boundary, coefficients, Jacobian) so as the resulting streamfunction remain a solution. We prove

\begin{theorem}[Deforming solutions of elliptic equations]\label{thm1}
 Let $\alpha\in (0,1)$ and $k\geq 2$. Fix two domains $D_0, D \subseteq \R^2$
with $C^{k,\alpha}$ boundaries given by \eqref{bdry} and a solution $\psi_0 : D_0 \to \mathbb{R}$ to
  \eqref{modelpbm} on $D_0$ with $\psi_0 \in C^{k,\alpha}(D_0)$ and $F_0 \in C^{k-1,\alpha}(\mathbb{R})$.  Suppose in addition that {\rm (\textbf{H1})}, {\rm (\textbf{H2})} and {\rm (\textbf{H3})} are satisfied.
  Let $\rho: D_0 \to \R$ such that $\rho \in C^{k-1,\alpha}(D_0)$ and $\int_{D_0} \rho = {\rm Vol}(D)$.
  Suppose that for sufficiently small $\ve > 0$, $|{\rm Vol}(D)- {\rm Vol}(D_0)| \lesssim \ve$ as well as
  \begin{alignat}{2}
  \| B- B_0\|_{C^{k,\alpha}} &\leq \ve  , \qquad\ \   \| \rho- 1\|_{C^{k-1,\alpha}} && \leq \ve\\
    \| a- a_0\|_{C^{k,\alpha}} &\leq \ve  , \qquad \ \  \| b- b_0\|_{C^{k,\alpha}} &&\leq \ve,\\
  \|G- G_0\|_{C^{k-2,\alpha}} &\leq \ve .
\end{alignat}
  Then, for $\ve$ sufficiently small, there exists a diffeomorphism $\gamma : D_0 \to D$ such that $\det \nabla \gamma =\rho$ and a function $F:\mathbb{R}\to \mathbb{R}$ satisfying $ \| F- F_0\|_{C^{k-2,\alpha}} \lesssim \ve$ such that
 the function $\psi = \psi_0\circ \gamma^{-1}$ satisfies
  the equation \eqref{modelpbmdef} in $D$.
  The diffeomorphism $\gamma$ is of the form $\gamma = {\rm id} +\nabla \eta
  +\nabla^\perp \phi$ and $\eta, \phi$ satisfy the estimates
  \begin{multline}
    \|\pa_s\eta\|_{C^{k,\alpha}} + \|\pa_s \phi\|_{C^{k,\alpha}}
     + \|\eta\|_{C^{k,\alpha}} + \|\phi \|_{C^{k,\alpha}}
     \\
     \leq C_{k,\alpha} \big( \|\rho  -1\|_{C^{k-1,\alpha}} +
     \|a - a_0\|_{C^{k,\alpha}} + \|b-b_0\|_{C^{k,\alpha}}
     + \|G - G_0\|_{C^{k-2,\alpha}} + \|B - B_0\|_{C^{k,\alpha}} \big)
   \label{gammabdthm}
  \end{multline}
  for constants $C_{k,\alpha}$ depending on $k, \alpha, D_0$,
  the ellipticity constant $M$ and $\|a_0\|_{C^{k,\alpha}}, \|b_0\|_{C^{k,\alpha}}$.
\end{theorem}

Theorem \ref{thm1} is used in \cite{CDG21} to construct approximate  solutions to the Magnetohydrostatic equations on wobbled tori which are nearly quasisymmetric. 
As in \cite{VW08}, one may think about these deformations arising dynamically  from a slow adiabatic deformation of the boundary, though we do not establish this point here.
We remark also that (\textbf{H1}) may not  be strictly needed for the Theorem \ref{thm1} provided kernel of $L_0 - \Lambda$ is very well understood. This is demonstrated in a slightly different context by the recent work \cite{CEW20} for Kolmogorov flow $u_0 = (\sin (y), 0)$ which is a shear with stagnation points so that Lemma \ref{kerlem} does not apply and the corresponding operator $\Delta-F'_0(\psi_0)$ has a non-trivial kernel (consisting of linear combinations of $\{\sin(y), \cos(y),\sin(x), \cos(x)\}$).  To deal with this degeneracy,  extra degrees of freedom are introduced in the contraction scheme.

We note that we can iterate the above
theorem to impose a nonlinear
constraint on $\rho$. Specifically, given a function
$X = X(y, \phi, \eta,  \nabla \phi, \nabla \eta,
\nabla \pa_s\phi, \nabla \pa_s\eta)$ with
$X|_{\phi, \eta = 0}$ sufficiently close to one, we can solve for the diffeomorphism $\gamma$ so that
$\rho = X$, at the expense of slightly modifying the domain.
Fixing notation, we consider
\begin{equation}
  X: D_0 \times \R \times \R \times \R^2 \times \R^2\times \R^2\times \R^2, \quad
 X = X(y, q_1, q_2, p_1, p_2, p_3, p_4)
 \label{Xsetup}
\end{equation}
and write
\begin{equation}
 X_0(y) = X|_{(q, p) = (0,0)}
 \qquad DX_0(y) = (\nabla_q X, \nabla_p X)|_{(q, p) = (0,0)}
 \label{X0notation}
\end{equation}
with $q = (q_1, q_2), p = (p_1, p_2, p_3, p_4)$.

\begin{theorem}[Deforming with an imposed nonlinear constraint]\label{thm2}
 Let $\alpha\in (0,1)$ and $k\geq 2$. Fix two domains $D_0, D \subseteq \R^2$
with $C^{k,\alpha}$ boundaries given by \eqref{bdry} and a solution $\psi_0 : D_0 \to \mathbb{R}$ to
  \eqref{modelpbm} on $D_0$ with $\psi_0 \in C^{k,\alpha}(D_0)$ and $F_0 \in C^{k-2,\alpha}(\mathbb{R})$.
   Let $X$ be as in \eqref{Xsetup} and satisfy $X\in C^{k,\alpha}$.

   Suppose that for sufficiently small $\ve, \ve_X > 0$, $|{\rm Vol}(D)- {\rm Vol}(D_0)| \leq \ve$ and that
    \begin{alignat}{2}
    \| B- B_0\|_{C^{k,\alpha}} &\leq \ve  , \qquad
    \| a- a_0\|_{C^{k,\alpha}} &&\leq \ve , \\
    \| b- b_0\|_{C^{k,\alpha}} &\leq \ve,\qquad
    \| G- G_0\|_{C^{k-2,\alpha}} &&\leq \ve,
\end{alignat}
and with notation as in \eqref{X0notation},
\begin{equation}
  \| X_0- 1\|_{C^{k,\alpha}}
    + \|DX_0\|_{C^{k-1,\alpha}} \leq \ve_X.\label{BXbds}
\end{equation}
  Then there exists a $\sigma>0$ and a diffeomorphism $\gamma:D_0 \to D_\sigma$ where $ D_\sigma  = \sigma D$
  is a dilation-by-$\sigma$ of the domain $D$ satisfying
  \be
\det \nabla \gamma=: \rho = X(y, \eta, \phi, \pa_s\phi, \pa_s\eta,
\nabla \pa_s\eta, \nabla \pa_s\phi),
\label{satisfynonlin}
  \ee
  with the dilation factor $\sigma > 0$ given by
$
  \sigma^2:=  {\rm Vol} D\bigg/ \int_{D_0} \rho.
$
Moreover, there is a function $F:\mathbb{R}\to \mathbb{R}$ satisfying $ \| F- F_0\|_{C^{k-2,\alpha}} \lesssim \ve$ such that
  $\psi = \psi_0\circ \gamma^{-1}$ is a solution to
  the  \eqref{modelpbmdef} in $D_\sigma$.
\end{theorem}

 We do not apply Theorem \ref{thm2} in the present paper. We record it here since it exploits a freedom in the construction and may be useful to build solutions with additional desirable properties (such as quasisymmetry in the context of plasma confinement fusion, see \cite{CDG21}).

\section{Applications to Fluid Systems }

\subsection{Proof of Theorem \ref{2dEthm} }

In the case of two-dimensional Euler equations on the channel, we apply Theorem \ref{thm1}   with $a_0^{ij} = a^{ij}\equiv \delta^{ij}$, $b_0^i=b^i\equiv 0$, $c_0=c=0$, $G=G_0=0$ and $F_0 =F_0$, the vorticity of the base state.
 As a result of Theorem \ref{liouE}, our base state $u_0= (v_0(y),0)$ is a shear where $v_0(y)= -\psi_0'(y)$ never vanishes.  As a consequence, it  satisfies \eqref{gseuler1} with $F_0(\psi)= \psi_0''(\psi_0^{-1}(\psi))$.
 We now show that all the hypotheses are met.
\\

\vspace{-2mm}
We first claim that in this setting  Hypothesis {\rm \textbf{(H1$^\prime$)}}  is a consequence
of the nondegeneracy of the base shear flow.
This follows immediately from the following Lemma (see e.g. \cite{GN17})
\begin{lemma}\label{kerlem}
Let $\Omega$ the periodic channel and let $u_0= (v_0(y_2),0)$ be a shear flow steady Euler solution and suppose $\inf_\Omega |v_0|>0$.  For all $u$ such that $u|_{\partial \Omega}=0$, the following holds
\be
 \int_\Omega  u\Big(\Delta   - F_0'(\psi_0)\Big) u\  \rmd y_1 \rmd y_2= -  \int_\Omega  | v_0(y_2)|^2 \left|\nabla \left(\frac{u}{| v_0(y_2)|}\right)\right|^2  \rmd y_1 \rmd y_2.
\ee
\end{lemma}
\begin{proof}
Note that  $F_0'(\psi_0(y_2))=(v_0''/v_0)(y_2)$. The result follows from direct computation.
\end{proof}

 Hypothesis {\rm \textbf{(H2)}} follows by our assumption that there are no stagnation points.
 \\

\vspace{-2mm}

\subsection{Proof of Theorem \ref{2dEthmarnold} }
In the case of two-dimensional Euler equations on $(M_0,g_0)$ we apply Theorem \ref{thm1}
$L_0 = \Delta_{g_0}$ and $L = \Delta_g$.  The coefficients can be computed directly in terms of the metrics and it is clear that the hypotheses on the closeness of $a_0$ and $a$ (resp $b_0$ and $b$) hold when $g_0$ is close to $g$.  Note also that under our hypotheses, $\L_\xi \psi=0$ according to Theorem \ref{propA}.
\\

\vspace{-2mm}
 Hypothesis {\rm \textbf{(H1$^\prime$)}}   follows by the assumption of Arnol'd stability.
\\

\vspace{-2mm}
 Hypothesis {\rm \textbf{(H2)}} follows by our assumption on the base states that they are non-degenerate.
 \\

\vspace{-2mm}
 Persistence of stability follows because Arnol'd stability conditions are open and our perturbation is small.

\begin{remark}[Hypothesis {\rm (\textbf{H3})}  on Domains with Symmetry]
We remark that if the domain admits a symmetry direction tangent to the boundary (so that all Arnol'd stable solutions enjoy the same symmetry according to Proposition \ref{propA}), we may apply Lemma \eqref{h3lem} to write explicitly the solution in Hypothesis {\rm (\textbf{H3})}.  Specifically, we apply the Lemma with $L_0 = \Delta_g$ and $\Lambda = F'$, and appeal to the following result
\begin{lemma}\label{h3lem2}
Let $\Delta_g$ be the Laplace--Beltrami operator on $(M,g)$. Suppose $\xi$ is a non-vanishing Killing field for $g$ which is tangent to $\partial M$.  Assume for $\psi\in C^{k,\alpha}(M)$ satisfies {\rm \textbf{(H2)}} and that $\L_\xi \psi=0$. Then for any function $f\in C^{k,\alpha}(\mathbb{R})$, we have  $\Delta_g f(\psi)= G(\psi)$ for some function $G\in C^{k-2,\alpha}(\mathbb{R})$.
\end{lemma}
\begin{proof}
First, by assumption the integral curves of $\xi$ foliate $M$.  Since $\L_\xi \psi=0$, we know that $\psi$ is constant on integral curves of $\xi$.  Moreover, since {\rm \textbf{(H2)}} guarantees that $|\nabla_g \psi|>0$ except at one point (if the domain is simply connected, and nowhere otherwise),  $\psi$ takes different values on different integral curves of $\xi$. Since $\xi$ is a  Killing field, $\L_\xi \Delta_g f( \psi)= \Delta_g ( f'(\psi) \L_\xi  \psi )=0$ and therefore $\L_\xi \Delta_g  f(\psi)$ is constant on integral curves of $\xi$ and thus a function of $\psi$.
\end{proof}
Therefore, Lemma \ref{h3lem} and  \ref{h3lem2} show that Hypothesis {\rm (\textbf{H3})}  holds with an explicit $u$ in the symmetric setting.
\end{remark}

\subsection{Proof of Theorem \ref{2dBthm} }

In the case of two-dimensional Boussinesq equations on the channel, we have $a_0^{ij} = a^{ij}\equiv \delta^{ij}$,  $b_0^i=b^i\equiv 0$, $c_0=c=0$,
 $G_0=y \Theta_0'(\psi_0)$, $G=y \Theta'(\psi)$,  and $F_0= G_0'$.
Then $L_0=\Delta$ and
$\Lambda =G_0'(\psi_0) + y_2 \Theta_0'(\psi_0)$.
\\

\vspace{-2mm}
 Hypothesis {\rm \textbf{(H1$^\prime$)}}  is verified for the following reason.
Since $\psi_0$ is a shear $\psi_0= \psi_0(y_2)$.  Given this, we know that $\Delta \psi_0=  \psi_0''(y_2)$ so that
$G_0(c) + y_2 \Theta_0(c)=  \psi_0''(\psi_0^{-1}(c))$.  Thus
\be
\Lambda= G_0'(\psi_0) + y_2 \Theta_0'(\psi_0)=\frac{v_0''(y_2)}{v_0(y_2)}
\ee
where $v_0= \psi_0'$. Thus Lemma \ref{kerlem} is applicable and the hypothesis follows.
\\

\vspace{-2mm}
 Hypothesis {\rm\textbf{(H2)}} follows by our assumption that there are no stagnation points.
 \\


\vspace{-2mm}
The result of the deformation defines a stream function $\psi$ for the Boussinesq equations with velocity $u=\nabla^\perp \psi$ and temperature profile $\theta = \Theta(\psi)$ (note $\Theta$ is recovered from $\Theta'$ up to a constant, which can be absorbed into the pressure).

\subsection{Proof of Theorem \ref{3dMHDthm} }

\vspace{-2mm}
 Hypothesis {\rm \textbf{(H1)}}  is verified when
 \be
{\rm im}\Big(   (C_0C_0')'(\psi_0)- r^2 P_0'(\psi_0)\Big) \notin {\rm Spec} \left(-\Delta+ \frac{1}{r} \partial_r\right).
 \ee
Since $-\Delta+ \tfrac{1}{r} \partial_r$ is a positive operator  Hypothesis {\rm \textbf{(H1$^\prime$)}}  is verified when
\be
(C_0C_0')'(\psi_0)- r^2 P_0'(\psi_0)< 0.
\ee
 Hypothesis {\rm\textbf{(H2)}} follows by our assumption that there are no stagnation points.

\appendix

\section{Proof of Lemma \ref{h1h2h3lem}}\label{h1h2h3lemsec}

We aim to solve $K_{\psi_0}[u] = g$ for a $g:= g(\psi_0)$ with $g(\partial D_0)=0$.  To avoid technical difficulties of defining the trace of a function which is just in $L^2$, we solve this equation assuming $g\in H^1$ for an $u\in H^{-1}$.
 Define the spaces
\begin{align}
 S^{(k)}  &= \{f\in H^k(D_0)\ | \  \mathbb{Q}_{\psi_0} f = 0\},\\
S^{(k)}_0 &= \{f\in H^k(D_0)\ | \  \mathbb{Q}_{\psi_0} f = 0, \ \  f|_{\partial D_0}=0\}.
\end{align}
Note that for $k\geq 1$ the operator $\mathbb{P}_{\psi_0}:  (H^{k}\cap H^{1}_0)(D_0) \to S^{(k)}_0$ is a continuous operator (which follows from \eqref{geolem}) and that therefore $S^{(k)}_0$ is a closed subspace of $H^{k}\cap H^{1}_0$.
We also remark that for all $f\in S^k$, we know that $\nabla^\perp\psi_0 \cdot \nabla f = 0$ and therefore by (a minor extension of) Lemma \ref{eoslemma}, there exists a function $F\in H^k(I)$ such that $f= F\circ \psi_0$.

Recall now that for $u\in S^{(k)}$ and with this notation we have
$
K_{\psi_0} u:=  \mathbb{P}_{\psi_0}  (L_0 - \Lambda)^{-1}_{\rm hbc}[u].
$
For $k\geq 1$, this operator $K_{\psi_0}: S^{(k-2)}\to S^{(k)}_0$ is continuous since $(L_0 - \Lambda)^{-1}_{\rm hbc}: H^{k-2}(D_0)\to (H^{k}\cap H^1_0)(D_0)$ is continuous by Hypothesis  {\rm (\textbf{H1}$^\prime$)} together with the fact that $\mathbb{P}_{\psi_0}$ is continuous.
We remark that for $ f\in S^{(-1)}$,  we have $ (L_0 - \Lambda)^{-1}_{\rm hbc}[f]\in H^1$.  Define now
\begin{align}
S_K &:= \{ K_{\psi_0} f \ | \  f\in S^{(-1)} \}  \subseteq S^{(1)}_0.
\end{align}

We aim to show that $S_K=S^{(1)}_0$, that is, $K_{\psi_0}: S^{(-1)}\to S^{(1)}_0$ is  onto.
Since $K_{\psi_0}:S^{(k-2)} \to S^{(k)}_0$ is continuous, $S_K$ is a closed subspace of $S^{(1)}_0$  in the $H^1(D_0)$ topology\footnote{By the continuity of $K_{\psi_0}$, it suffices to prove that if $K_{\psi_0}[f^n]\to g$ in $H^1$ with $f^n\in H^{-1}$, then $f^n$ converges in $H^{-1}$.   By \eqref{coersive} we have
\be
c_0\|  f^n - f^m\|_{H^{-1}}^2  \leq \langle f^n - f^m , K_{\psi_0}[f^n - f^m]\rangle_{L^2}  \leq  C\|  f^n - f^m\|_{H^{-1}}\|K_{\psi_0}[f^n - f^m]\|_{H^1},
\ee
which can be justified by an approximation.
From this, we conclude that the sequence $\{f^n\}_{n\geq 0}$ is Cauchy in $H^{-1}$.
}.   Thus,  $S^{(1)}_0=S_K\oplus S_K^\perp$ where
   \begin{align}\label{Skperp}
  S_K^\perp &:= \{ f\in S^{(1)}_0  \ | \ \langle f,g\rangle_{S^{(1)}}= 0  \  \text{for all} \ g\in S_K\},
\end{align}
where the $S^{(1)}$ topology is equivalent to the $H^1$ topology and will be defined shortly.
Note that  for any $h\in S^\perp_K\subset S^{(1)}$, the function $K_{\psi_0}h \in S_K$, it follows from \eqref{Skperp} that  $\langle h,K_{\psi_0}h \rangle_{S^{(1)}}=0$.   Thus, to conclude that $S_K^\perp = \{ 0\}$, we show that $K_{\psi_0}$ has a trivial kernel in $S^{(1)}_0$.  This is accomplished by designing  a topology on $ S^{(1)}$ which is equivalent to the $H^1$ topology and showing that for all $h\in S^{(1)}$,  there is a constant $c>0$ depending only on $\psi_0$ such that
\be
\langle h, K_{\psi_0} h\rangle_{S^{(1)}} \geq c \| h\|_{L^2(D_0)}^2, \qquad \forall h\in S_0^{(1)}.
\ee
We now design the topology.
 First, we equip $S^{(0)}$ with the $L^2(D_0)$ topology:
\be
\langle f,g\rangle_{S^{(0)}} = \int_{D_0} f g.
\ee
Note that, using orthogonality of the projection \eqref{ortho}, by  Hypothesis  {\rm (\textbf{H1}$^\prime$)} we have
\begin{align}\label{coersive}
\langle h, K_{\psi_0} h\rangle_{S^{(0)}} :=\langle h, \mathbb{P}_{\psi_0} (L_0 - \Lambda)^{-1}_{\rm hbc}h \rangle_{S^{(0)}} &= \langle h, (L_0 - \Lambda)^{-1}_{\rm hbc}h \rangle_{L^2} \geq c_0\| h\|_{\dot{H}^{-1}}^2\geq 0
\end{align}
for some $c_0>0$
where $ \| h\|_{\dot{H}^{-1}}= \|\nabla g\|_{L^2}$ where $\Delta g = h$ and $g=0$ at the boundary.

Now let $\partial_{\psi_0} = \frac{\nabla \psi_0}{|\nabla \psi_0|^2} \cdot \nabla$. Recalling for any $h\in S^k$ there exists  $H\in H^k(I)$ such that $h=H\circ \psi_0$, we note that $\partial_{\psi_0} h= H'(\psi_0)$. Now let
\be
\langle f,g\rangle_{S^{(1)}} =  \langle\partial_{\psi_0}  f,\partial_{\psi_0} g\rangle_{S^{(0)}} + M\langle f,g\rangle_{S^{(0)}}.
\ee
for some large constant $M$ to be specified shortly. This topology is obviously equivalent to that of $H^1$ on $S^1$.  To see this, denoting $\hat{x}= x/|x|$, we can write $\nabla =\widehat{\nabla^{\perp}\psi_0} \widehat{\nabla^{\perp}\psi_0} \cdot \nabla +   \widehat{\nabla\psi_0}  \widehat{\nabla\psi_0} \cdot \nabla$ and notice that on any $f\in S^1$, $\nabla f = {\nabla\psi_0}   \partial_{\psi_0} f$.
 Now note that, since $h=0$ and $ \mathbb{P}_{\psi_0}  (L_0 - \Lambda)^{-1}_{\rm hbc}h=0$ on the boundary, we have
\begin{align}
\langle h, \mathbb{P}_{\psi_0} (L_0 - \Lambda)^{-1}_{\rm hbc}h \rangle_{S^{(1)}} &=\langle \partial_{\psi_0} h, \partial_{\psi_0} \mathbb{P}_{\psi_0} (L_0 - \Lambda)^{-1}_{\rm hbc}h \rangle_{S^{(0)}}+ M \langle h, \mathbb{P}_{\psi_0} (L_0 - \Lambda)^{-1}_{\rm hbc}h \rangle_{S^{(0)}}\\
&=-\langle \partial_{\psi_0}^2 h, \mathbb{P}_{\psi_0} (L_0 - \Lambda)^{-1}_{\rm hbc}h \rangle_{S^{(0)}}+ M \langle h, (L_0 - \Lambda)^{-1}_{\rm hbc}h \rangle_{S^{(0)}}\\
&=\langle \partial_{\psi_0} h, \partial_{\psi_0} (L_0 - \Lambda)^{-1}_{\rm hbc}h \rangle_{S^{(0)}}+ M \langle h,  (L_0 - \Lambda)^{-1}_{\rm hbc}h \rangle_{S^{(0)}}.
\end{align}
{In the above we repeatedly used that $\mathbb{P}_{\psi_0}$ is an orthogonal projection on $L^2$ so that $\langle f, \mathbb{P}_{\psi_0} g\rangle_{S^{(0)}}=\langle \mathbb{P}_{\psi_0} f, g\rangle_{S^{(0)}}$. Thus, when paired with functions only of $\psi_0$ such as  $h$ or $\partial_{\psi_0}^2 h$, the projector is the identity.}  Introducing $f=(L_0 - \Lambda)^{-1}_{\rm hbc}h$, we have
\begin{align}
\langle   \partial_{\psi_0} h,   \partial_{\psi_0} (L_0 - \Lambda)^{-1}_{\rm hbc}h \rangle_{S^{(0)}} &= \langle \partial_{\psi_0} (L_0 - \Lambda) f ,   \partial_{\psi_0} f \rangle_{S^{(0)}} =- \langle  (L_0 - \Lambda) f ,   \partial_{\psi_0}^2 f \rangle_{S^{(0)}}
\end{align}
since $ (L_0 - \Lambda) f=h$ which is zero at the boundary.  Now, by (c.f. \S 7.2, pg 390 of \cite{E10}) we have
\be
\langle  (L_0 - \Lambda) f ,   \partial_{\psi_0}^2 f \rangle_{S^{(0)}} \geq\beta  \| f\|_{H^2}^2 -  \gamma \| f\|_{L^2}^2,
\ee
for some constants $\beta,\gamma>0$.
Moreover, by our hypotheses, we have for some constant $c>0$ that
\be
 c \|  f\|_{H^1}^2\leq \langle   (L_0 - \Lambda)   f , f \rangle_{S^{(0)}}.
\ee
Combining the above bounds we obtain
\begin{align}
\langle h, \mathbb{P}_{\psi_0} (L_0 - \Lambda)^{-1}_{\rm hbc}h \rangle_{S^{(1)}} & \geq \beta \| f\|_{H^2}^2 - \gamma \| f\|_{L^2}^2+  cM \|  f\|_{H^1}^2.
\end{align}
 It follows that by choosing $M$ sufficiently large that for some $c_1>0$ depending only on $\psi_0$ we have
\begin{align}
\langle h, \mathbb{P}_{\psi_0} (L_0 - \Lambda)^{-1}_{\rm hbc}h \rangle_{S^{(1)}} &\geq c_1\| f\|_{H^2}^2
\end{align}
and we deduce $\langle h, \mathbb{P}_{\psi_0} (L_0 - \Lambda)^{-1}_{\rm hbc}h \rangle_{S^{(1)}}$ is coercive.

Thus we have established that for all $g\in S_0^{(1)}$, there exists a unique $u\in S^{(-1)}$ such that
\be
K_{\psi_0} [u] = g.
\ee
Now we want to show that for $k\geq 1$,  if $g\in S_0^{(k)}$, then  $u\in S^{(k-2)}$.  Let  $g\in S_0^{(2)}$.  We know there is a solution $u\in S_0^{(-1)}$.  We wish to show that actually   $u\in S_0^{(0)}$.  To see this, we formally differentiate:
\be
 \partial_{\psi_0} K_{\psi_0} [u] = \partial_{\psi_0} g.
\ee
The following formal apriori calculation can be made rigorous by an approximation argument.
We compute the commutator of derivative with $K_{\psi_0} $:
\begin{align}
[ \partial_{\psi_0}, K_{\psi_0}] f &=  \partial_{\psi_0}  K_{\psi_0}[f] -   K_{\psi_0}[ \partial_{\psi_0}f] \\
&=  \partial_{\psi_0}  \mathbb{P}_{\psi_0} (L_0 - \Lambda)^{-1}_{\rm hbc} f-   \mathbb{P}_{\psi_0} (L_0 - \Lambda)^{-1}_{\rm hbc} \partial_{\psi_0} f\\
&=  \mathbb{P}_{\psi_0}  \partial_{\psi_0}  (L_0 - \Lambda)^{-1}_{\rm hbc} f-   \mathbb{P}_{\psi_0} (L_0 - \Lambda)^{-1}_{\rm hbc} \partial_{\psi_0} f +  [\mathbb{P}_{\psi_0},  \partial_{\psi_0}]  (L_0 - \Lambda)^{-1}_{\rm hbc} f\\
&=  \mathbb{P}_{\psi_0}  [\partial_{\psi_0},  (L_0 - \Lambda)^{-1}_{\rm hbc}] f +  [\mathbb{P}_{\psi_0},  \partial_{\psi_0}]  (L_0 - \Lambda)^{-1}_{\rm hbc} f.
\end{align}
Now note that by Lemma \ref{geolem} we have
\be
[\mathbb{P}_{\psi_0},  \partial_{\psi_0}]  g = -\mathbb{P}_{\psi_0} \Big[ \left(\frac{\mu'}{\mu^2} +  \frac{\Delta\psi_0- 2 \kappa |\nabla \psi_0|}{|\nabla \psi_0|^2} \right) g  \Big].
\ee
Thus,  commutator of derivative with streamline projector is of zero order:
\be
\|[\mathbb{P}_{\psi_0},  \partial_{\psi_0}]  g\|_{L^2} \leq C \|g\|_{L^2}.
\ee
Also the commutator of derivative and the inverse operator is zero smoothing of degree -2. Specifically, note that with $f= (L_0 - \Lambda)^{-1}_{\rm hbc} g$ we have
\begin{align}
[\partial_{\psi_0},  (L_0 - \Lambda)^{-1}_{\rm hbc}] g  &= \partial_{\psi_0} f-  (L_0 - \Lambda)^{-1}_{\rm hbc} \partial_{\psi_0} (L_0 - \Lambda)f\\
&=   (L_0 - \Lambda)^{-1}_{\rm hbc} [\partial_{\psi_0},L_0 - \Lambda] f.
\end{align}
The commutator $ [\partial_{\psi_0},L_0 - \Lambda]$ is a differential operator of order 2. Thus we obtain
\be
\| [\partial_{\psi_0},  (L_0 - \Lambda)^{-1}_{\rm hbc}] g \|_{L^2} \leq  C \| (L_0 - \Lambda)^{-1}_{\rm hbc} g\|_{L^2},
\ee
and we obtain the estimate
\be
\|[ \partial_{\psi_0}, K_{\psi_0}] f \|_{L^2} \leq C  \| (L_0 - \Lambda)^{-1}_{\rm hbc} f\|_{L^2}.
\ee
Moreover, $[ \partial_{\psi_0}, K_{\psi_0}] f$ is zero on the boundary.
Then if $u\in S^{(-1)}$ and $g\in S_0^{(2)}$
\be
 K_{\psi_0} [ \partial_{\psi_0} u]   = [ \partial_{\psi_0}, K_{\psi_0}]u+  \partial_{\psi_0} g\in S_0^{(1)}.
\ee
It follows that $ \partial_{\psi_0} u\in S^{(-1)}$ and thus $  u\in S^{(0)}$.  Higher regularity follows by similar arguments.

\section{Proof of Proposition \ref{pset}}
\label{changeofcoords}

Let $D_0, D$ be two nearby domains and let $\gamma:D_0 \to D$ be a diffeomorphism.
Denote points in $D$ by $(x_1,x_2)$ and points in $D_0$ by $(y_1,y_2)$.
Decompose the diffeomorphism
\be
\gamma ={\rm id} + (\alpha, \beta) =
{\rm id } + \nabla^\perp \phi + \nabla \eta.
\ee
where $\nabla^\perp = (-\partial_2, \partial_1)$.
Let $\rho = \det \nabla \gamma$ and
$0<|\rho| < \infty$.
Write
$\psi = \psi_0\circ \gamma^{-1}$.  More explicitly
\be
\xx = \yx+ \alpha (\yx, \yy), \qquad \xy = \yy+ \beta (\yx, \yy).
\ee
and
\be
\alpha (\yx, \yy) = - \pa_\yy \phi +\pa_\yx \eta, \qquad \beta (\yx, \yy) =  \pa_\yx \phi +\pa_\yy \eta.
\ee
We have so $\nabla_y \psi_0 = \nabla \gamma \cdot (\nabla_x \psi) \circ \gamma, $ and
\begin{align}
(\nabla_x \psi) \circ \gamma &= ( \nabla \gamma )^{-1}\cdot  \nabla_y \psi_0, \qquad  \nabla \gamma =I +  \begin{pmatrix} \pa_\yx \alpha  &  \pa_\yx \beta    \\  \pa_\yy \alpha & \pa_\yy \beta    \end{pmatrix}.
\end{align}

 \noindent \textbf{Jacobian of Transformation:}
 Note that, with $\rho = \det \nabla \gamma$ we find
\begin{align}
\rho &= 1 + \pa_\yx \alpha  +\pa_\yy \beta +  ( \pa_\yx \alpha  \pa_\yy \beta   -   \pa_\yy \alpha   \pa_\yx \beta   )= 1 + \Delta \eta  - \sN_\eta(\partial^2 \eta, \partial^2\phi)\label{etaeq}
\end{align}
where
\be\label{Neta}
\sN_\eta :=  \pa_\yx \alpha  \pa_\yy \beta   -   \pa_\yy \alpha   \pa_\yx \beta =  - \nabla  \alpha  \cdot  \nabla^\perp \beta.
\ee \\
\vspace{-2mm}
 \noindent \textbf{Inverse Gradient:}
A useful expression for the inverse gradient of the transformation is

\begin{align}
( \nabla \gamma )^{-1}
&= \frac{1}{\rho}\left( I + \begin{pmatrix} \pa_\yy \beta &-  \pa_\yx \beta    \\ - \pa_\yy \alpha &  \pa_\yx \alpha    \end{pmatrix}\right)=  \frac{1}{\rho}\left(I+ \begin{pmatrix} \pa_\yy  \pa_\yx \phi +\pa_\yy^2 \eta &-   \pa_\yx^2 \phi - \pa_\yx\pa_\yy \eta    \\
\pa_\yy^2 \phi -\pa_\yx \pa_\yy  \eta &  - \pa_\yy\pa_\yx  \phi +\pa_\yx^2 \eta    \end{pmatrix} \right)\\
& = \frac{1}{\rho}\left(I+ \begin{pmatrix} \pa_\yy  \pa_\yx \phi &-   \pa_\yx^2 \phi     \\
\pa_\yy^2 \phi&  - \pa_\yy\pa_\yx  \phi     \end{pmatrix}+ \begin{pmatrix} \pa_\yy^2 \eta &- \pa_\yx\pa_\yy \eta    \\
 -\pa_\yx \pa_\yy  \eta &   \pa_\yx^2 \eta    \end{pmatrix} \right)  \\
 & = \frac{1}{\rho}\left(I- \begin{pmatrix} \pa_\yx \nabla_y^\perp \phi    \\
 \pa_\yy \nabla_y^\perp \phi     \end{pmatrix}
 +\begin{pmatrix} -\pa_\yy \nabla_y^\perp \eta    \\
\pa_\yx \nabla_y^\perp \eta     \end{pmatrix} \right).
\end{align}\\

\vspace{-2mm}
 \noindent \textbf{Derivatives of $\psi:= \psi_0 \circ \gamma^{-1}$:}
In the above, $ \pa_\yx \nabla_y^\perp \phi $ and similar terms are understood as row vectors forming the matrices.
Thus
\begin{align}
(\nabla_x \psi)\circ \gamma
 & = \frac{1}{\rho}\left(\nabla_y \psi_0+   \nabla_y \partial_s \phi + (\nabla_y^\perp \phi \cdot \nabla_y) \nabla_y \psi_0
 \right)   - \frac{1}{\rho}\Big(  \nabla_y^\perp \partial_s \eta +  (\nabla_y \eta \cdot \nabla_y) \nabla_y \psi_0\Big),
\end{align}
where we introduced the notation for streamline derivatives $ \partial_s = \nabla^\perp\psi_0\cdot \nabla_y$.  In the future, we will bin terms involving $\eta$ in
\be\label{L0}
\sL_0(\partial^2 \eta,\rho;\partial^2\psi_0):=  - \frac{1}{\rho}\Big(  \nabla_y^\perp \partial_s \eta +  (\nabla_y \eta \cdot \nabla_y) \nabla_y \psi_0\Big).
\ee
Note we track only the highest number derivatives in the notation on the left.
We now obtain a formula for the Hessian in terms of $\psi_0$ and the diffeomorphism.  First note that
\be
 (\nabla_x \otimes \nabla_x \psi)\circ \gamma = (\nabla_y\gamma)^{-1} \nabla_y \Big((\nabla_x \psi)\circ \gamma \Big).
\ee
The right-hand-side of the above is calculated as
\begin{align}
\nabla_y \Big((\nabla_x \psi)\circ \gamma \Big)&= \rho \nabla_y \rho^{-1}\otimes ( \nabla_x \psi)\circ \gamma\\
&\qquad  + \frac{1}{\rho}\Bigg(\nabla_y \otimes \nabla_y \psi_0+(\nabla_y \otimes  \nabla_y) \partial_s \phi + (\nabla_y \otimes \nabla_y^\perp \phi) (\nabla_y\otimes  \nabla_y \psi_0)\\
&\qquad\qquad\qquad  + ( \nabla_y^\perp \phi \cdot \nabla)\nabla_y\otimes \nabla_y \psi_0
 -(\nabla_y\otimes \nabla_y^\perp) \partial_s \eta\\
 & \qquad\qquad\qquad\qquad -  (\nabla_y \otimes \nabla_y \eta ) (\nabla_y\otimes \nabla_y \psi_0) - (\nabla_y \eta \cdot \nabla) \nabla_y\otimes \nabla_y \psi_0 \Bigg).
\end{align}
Thus we obtain
\begin{align}
 (\nabla_x \otimes \nabla_x \psi)\circ \gamma
&=  \frac{1}{\rho^2}\nabla_y \otimes \nabla_y \psi_0\ + \frac{1}{\rho^2}\Bigg((\nabla_y \otimes  \nabla_y) \partial_s \phi -( \nabla_y \phi \cdot \nabla^\perp)\nabla_y\otimes \nabla_y \psi_0\Bigg) \\
&\qquad + \sL_1(\partial^3 \eta,\partial \rho;\partial^3 \psi_0)+  \sN_1(\partial^2\phi,\partial^2\partial_s \phi,\partial^3 \eta,\pa \rho;\partial^3 \psi_0),
\end{align}
where we have grouped the terms linear in $\eta$ and $\nabla \rho$
\begin{align}
\sL_1(\partial^3 \eta,\partial \rho;\partial^3 \psi_0)&:=   - \frac{1}{\rho^2}\Bigg( \nabla_y \rho \otimes ( \nabla_x \psi_0)\circ \gamma  +(\nabla_y\otimes \nabla_y^\perp) \partial_s \eta + (\nabla_y \eta \cdot \nabla) \nabla_y\otimes \nabla_y \psi_0\Bigg),\label{Ldef}
\end{align}
as well as
all terms which are quadratic in combinations of $(\phi,\eta, \nabla \rho)$:
\begin{align}
\sN_1(\partial^2\phi,&\partial^2\partial_s \phi,\partial^3 \eta,\pa \rho;\partial^3 \psi_0)\\
&:= -\frac{1}{\rho^2} \nabla_y \rho\otimes   \left(\nabla_y \partial_s \phi + (\nabla_y^\perp \phi \cdot \nabla_y) \nabla_y \psi_0
 -  \nabla_y^\perp \partial_s \eta -  (\nabla_y \eta \cdot \nabla_y) \nabla_y \psi_0 \right)  \\
&\qquad +\frac{1}{\rho^2} \left(- \begin{pmatrix} \pa_\yx \nabla_y^\perp \phi    \\
 \pa_\yy \nabla_y^\perp \phi     \end{pmatrix}
 +\begin{pmatrix} -\pa_\yy \nabla_y^\perp \eta    \\
\pa_\yx \nabla_y^\perp \eta     \end{pmatrix} \right)\Bigg[ -  \nabla_y \rho \otimes ( \nabla_x \psi_0)\circ \gamma +  \\
&\qquad -  \nabla_y \rho\otimes \Big(  \nabla_y \partial_s \phi + (\nabla_y^\perp \phi \cdot \nabla_y) \nabla_y \psi_0
 -  \nabla_y^\perp \partial_s \eta -  (\nabla_y \eta \cdot \nabla_y) \nabla_y \psi_0\Big) \\
&\qquad + \Bigg((\nabla_y \otimes  \nabla_y) \partial_s \phi + (\nabla_y \otimes \nabla_y^\perp \phi) (\nabla_y\otimes  \nabla_y \psi_0)+ ( \nabla_y^\perp \phi \cdot \nabla)\nabla_y\otimes \nabla_y \psi_0
\\
 & \qquad \  \ \  -(\nabla_y\otimes \nabla_y^\perp) \partial_s \eta-  (\nabla_y \otimes \nabla_y \eta ) (\nabla_y\otimes \nabla_y \psi_0) - (\nabla_y \eta \cdot \nabla) \nabla_y\otimes \nabla_y \psi_0 \Bigg)\Bigg].\label{Ndef}
\end{align}
We note the important point the nonlinearity involves third derivatives of $\phi$ only through $\partial^2\partial_s \phi$.
We now introduce stream function coordinates.
Note the following formula for
$\nabla_y$ in terms of derivative along and transverse to streamlines, i.e.
$\partial_s = \nabla^\perp \psi_0\cdot\nabla$ and $\partial_{\psi_0} = \nabla\psi_0\cdot
\nabla$,
\be
\nabla_y =  \frac{1}{|\nabla \psi_0|^2} \Big[\nabla \psi_0 \partial_{\psi_0} +\nabla^\perp \psi_0 \partial_s \Big], \qquad \nabla_y^\perp=  \frac{1}{|\nabla \psi_0|^2} \Big[\nabla^\perp \psi_0 \partial_{\psi_0} -\nabla \psi_0 \partial_s \Big].
\ee
With this, we have
\begin{align}
 \nabla_y \phi \cdot \nabla^\perp 
 &=\frac{1}{|\nabla \psi_0|^2} \Big(( \partial_s \phi) \partial_{\psi_0} - (\partial_{\psi_0} \phi) \partial_s \Big).
\end{align}
We arrive at
\begin{lemma} The following formulae hold
\begin{align}
(\nabla_x \psi)\circ \gamma
 & =\frac{1}{\rho}\nabla_y \psi_0 +  \frac{1}{\rho}\Bigg(  \nabla_y \partial_s \phi + \frac{1}{|\nabla \psi_0|^2} \Big[( \partial_s \phi) \partial_{\psi_0} - (\partial_{\psi_0} \phi) \partial_s \Big] \nabla_y \psi_0 \Bigg)+  \sL_0(\partial^3 \eta, \rho;\partial^3 \psi_0),\\
 (\nabla_x \otimes \nabla_x \psi)\circ \gamma&=  \frac{1}{\rho^2}\nabla_y \otimes \nabla_y \psi_0  -\frac{1}{\rho^2 |\nabla \psi_0|^2}  \Big[( \partial_s \phi) \partial_{\psi_0} - (\partial_{\psi_0} \phi) \partial_s  \Big]\nabla_y\otimes \nabla_y \psi_0\\
 &\qquad  + \frac{1}{\rho^2} (\nabla_y \otimes  \nabla_y) \partial_s \phi   + \sL_1(\partial^3 \eta,\partial \rho;\partial^3 \psi_0)+   \sN_1(\partial^2\phi,\partial^2\partial_s \phi,\partial^3 \eta,\pa \rho;\partial^3 \psi_0),
\end{align}
where $\sL_0$, $\sL_1$ and $\sN_1$ are defined by \eqref{L0}, \eqref{Ldef} and \eqref{Ndef}.  Let $a_\gamma = a\circ \gamma$, $b_\gamma= b\circ \gamma$ and
\be
 L:= a: \nabla_x\otimes \nabla_x + b\cdot \nabla_x, \qquad   L_\gamma = a_\gamma: \nabla_y\otimes \nabla_y + b_\gamma \cdot \nabla_y.
\ee
Then we have
\begin{align}
(  L \psi )\circ \gamma &= \frac{1}{\rho^2} L_\gamma \psi_0 +  \frac{1}{\rho^2} L_\gamma \partial_s \phi -\frac{1}{\rho^2 |\nabla \psi_0|^2}\Big[( \partial_s \phi) \partial_{\psi_0} - (\partial_{\psi_0} \phi) \partial_s  \Big]L_\gamma \psi_0\\
&\qquad + \frac{1}{\rho^2 |\nabla \psi_0|^2}\Bigg( \Big[( \partial_s \phi) \partial_{\psi_0}a_\gamma - (\partial_{\psi_0} \phi) \partial_s a_\gamma \Big]: \nabla_y\otimes \nabla_y  \psi_0+ \Big[( \partial_s \phi) \partial_{\psi_0}b_\gamma - (\partial_{\psi_0} \phi) \partial_s b_\gamma \Big]\cdot  \nabla_y  \psi_0\Bigg)\\
&\qquad \qquad+a: \sL_1(\partial^3 \eta,\partial \rho;\partial^3 \psi_0)+  a: \sN_1(\partial^2\phi,\partial^2\partial_s \phi,\partial^3 \eta,\pa \rho;\partial^3 \psi_0) + b \cdot\sL_0(\partial^3 \eta,\partial \rho;\partial^3 \psi_0).
\end{align}
\end{lemma}
We now simplify these formulae in the setting where $\psi_0$ and $\psi$ satisfy \eqref{modelpbm}, \eqref{modelpbmdef}.  Recall
$
 L_0:= a_0: \nabla_y\otimes \nabla_y + b_0\cdot \nabla_y,
$
so that
\begin{align}
L_\gamma- L_0 &=  (a_\gamma -a_0): \nabla_y\otimes \nabla_y + (b_\gamma - b_0)\cdot \nabla_y\\
&=   (a-a_0)_\gamma: \nabla_y\otimes \nabla_y + (b-b_0)_\gamma \cdot \nabla_y + ((a_0)_\gamma -a_0): \nabla_y\otimes \nabla_y + ((b_0)_\gamma - b_0)\cdot \nabla_y.
\end{align}
We now denote the nonlinearities arising in expanding by
\be
\sR_{a_0}(\partial \phi, \partial \eta, \partial^2 a_0) :=(a_0)_\gamma -a_0 - (\gamma-y) \cdot \nabla a_0,
\qquad
 \sR_{b_0}(\partial \phi, \partial \eta, \partial^2 b_0) :=(b_0)_\gamma -b_0 - (\gamma-y) \cdot \nabla b_0.
\ee
Note that the dependences are a consequence of Taylor's formula.  Then
\begin{align}
L_\gamma- L_0 &=     (a-a_0)_\gamma: \nabla_y\otimes \nabla_y + (b-b_0)_\gamma \cdot \nabla_y\\
&\qquad +(\gamma- y) \cdot \nabla a_0: \nabla_y\otimes \nabla_y +  (\gamma- y) \cdot \nabla b_0\cdot \nabla_y\\
&\qquad\qquad + \sR_{a_0}(\partial \phi, \partial \eta, \partial^2 a_0)  : \nabla_y\otimes \nabla_y +  \sR_{b_0}(\partial \phi, \partial \eta, \partial^2 b_0)\cdot \nabla_y\\
&=    (a-a_0)_\gamma: \nabla_y\otimes \nabla_y  + (b-b_0)_\gamma \cdot \nabla_y\\
&\qquad +  \nabla^\perp \phi  \cdot \nabla a_0: \nabla_y\otimes \nabla_y +  \nabla^\perp \phi  \cdot \nabla  b_0\cdot \nabla_y\\
&\qquad\qquad +\nabla \eta \cdot \nabla a_0: \nabla_y\otimes \nabla_y +    \nabla \eta\cdot \nabla b_0\cdot \nabla_y\\
&\qquad\qquad \qquad+ \sR_{a_0}(\partial \phi, \partial \eta, \partial^2 a_0)  : \nabla_y\otimes \nabla_y +  \sR_{b_0}(\partial \phi, \partial \eta, \partial^2 b_0)\cdot \nabla_y.
\end{align}
Thus we have
\begin{align}
(L_\gamma- L_0 )\psi_0 &=  - \frac{1}{|\nabla \psi_0|^2} \left(( \partial_s \phi) \partial_{\psi_0}a_0 - (\partial_{\psi_0} \phi) \partial_sa_0 \right) : \nabla_y\otimes \nabla_y\psi_0 \\
&\qquad - \frac{1}{|\nabla \psi_0|^2} \left(( \partial_s \phi) \partial_{\psi_0}b_0 - (\partial_{\psi_0} \phi) \partial_sb_0 \right)\cdot \nabla_y\psi_0\\
&\qquad \qquad+  \sL_2(\delta a, \delta b,\partial a_0, \partial b_0, \partial \eta,\gamma; \partial^2 \psi_0) + \sN_2(\partial^2 a_0, \partial^2 b_0, \partial \phi, \partial \eta, \gamma; \partial^2 \psi_0),
\end{align}
where
\begin{align}
\sL_2(\delta a, \delta b,\partial a_0, \partial b_0, \partial \eta,\gamma; \partial^2 \psi_0) &:=   (a-a_0)_\gamma: \nabla_y\otimes \nabla_y \psi_0 + (b-b_0)_\gamma \cdot \nabla_y\psi_0\\
&\qquad\qquad +\nabla \eta \cdot \nabla a_0: \nabla_y\otimes \nabla_y\psi_0 +    \nabla \eta\cdot \nabla b_0\cdot \nabla_y\psi_0,\\
\sN_2(\partial^2 a_0, \partial^2 b_0, \partial \phi, \partial \eta,\gamma; \partial^2 \psi_0) &:= \sR_{a_0}(\partial \phi, \partial \eta, \partial^2 a_0)  : \nabla_y\otimes \nabla_y \psi_0+  \sR_{b_0}(\partial \phi, \partial \eta, \partial^2 b_0)\cdot \nabla_y\psi_0.
\end{align}
Additionally we find
\begin{align}
(L_\gamma- L_0 )\partial_s \phi &=  \sL_3(\delta a, \delta b,\partial a_0, \partial b_0, \partial \eta,\gamma, \partial^2 \partial_s\phi)  +\sN_3(\partial^2 a_0, \partial^2 b_0, \partial \phi, \partial \eta,\gamma, \partial^2 \partial_s \phi).
\end{align}
where
\begin{align}
\sL_3(\delta a, \delta b,\partial a_0, \partial b_0, \partial \eta,\gamma, \partial^2 \partial_s\phi) &:=   (a-a_0)_\gamma: \nabla_y\otimes \nabla_y \partial_s\phi + (b-b_0)_\gamma \cdot \nabla_y\partial_s\phi\\
\sN_3(\partial^2 a_0, \partial^2 b_0, \partial \phi, \partial \eta,\gamma, \partial^2 \partial_s \phi) &:=\nabla^\perp \phi  \cdot \nabla a_0: \nabla_y\otimes \nabla_y\partial_s \phi  +  \nabla^\perp \phi  \cdot \nabla  b_0\cdot \nabla_y\partial_s \phi \\
&\qquad +\nabla \eta \cdot \nabla a_0: \nabla_y\otimes \nabla_y\partial_s \phi  +    \nabla \eta\cdot \nabla b_0\cdot \nabla_y\partial_s \phi \\
&\qquad + \sR_{a_0}(\partial \phi, \partial \eta, \partial^2 a_0)  : \nabla_y\otimes \nabla_y\partial_s \phi +  \sR_{b_0}(\partial \phi, \partial \eta, \partial^2 b_0)\cdot \nabla_y\partial_s \phi .
\end{align}
Thus we have
\begin{align}
\rho^2 (  L \psi )\circ \gamma &=L_0 \psi_0 +  L_0 \partial_s \phi- \frac{1}{|\nabla \psi_0|^2} \left(( \partial_s \phi) \partial_{\psi_0}a_0 - (\partial_{\psi_0} \phi) \partial_sa_0 \right) : \nabla_y\otimes \nabla_y\psi_0 \\
& \ \  - \frac{1}{|\nabla \psi_0|^2} \left(( \partial_s \phi) \partial_{\psi_0}b_0 - (\partial_{\psi_0} \phi) \partial_sb_0 \right)\cdot \nabla_y\psi_0 -\frac{1}{ |\nabla \psi_0|^2}\Big[( \partial_s \phi) \partial_{\psi_0} - (\partial_{\psi_0} \phi) \partial_s  \Big]L_0 \psi_0\\
& \ \  + \frac{1}{ |\nabla \psi_0|^2} \Big[( \partial_s \phi) \partial_{\psi_0}a - (\partial_{\psi_0} \phi) \partial_s a \Big]: \nabla_y\otimes \nabla_y  \psi_0+ \frac{1}{ |\nabla \psi_0|^2} \Big[( \partial_s \phi) \partial_{\psi_0}b - (\partial_{\psi_0} \phi) \partial_s b \Big]\cdot  \nabla_y  \psi_0\\
& \qquad+\sL_5(\delta a, \delta b,\partial a_0, \partial b_0, \partial^3 \eta,\gamma, \partial^2 \partial_s\phi; \partial^3 \psi_0) + \sN_5 (\partial^2 a_0, \partial^2 b_0, \partial^2\phi, \partial^2 \eta,\partial\rho, \partial^2 \partial_s \phi; \partial^3 \psi_0)
\end{align}
where the linear and nonlinear terms  are
\be\label{LN5}
\sL_5= \sum_{i=0}^4\sL_i, \qquad \sN_5 = \sum_{i=1}^4\sN_i.
\ee
Rearranging this, we have
\begin{align}
\rho^2 (  L \psi )\circ \gamma &=L_0 \psi_0 +  (L_0 - \lambda_1) \partial_s \phi + \lambda_2 \partial_{\psi_0} \phi\\
&\qquad +\sL_5(\delta a, \delta b,\partial a_0, \partial b_0, \partial^3 \eta, \partial\rho, \partial^2 \partial_s\phi; \partial^3 \psi_0) + \sN_5(\partial^2 a_0, \partial^2 b_0, \partial^2\phi, \partial^2 \eta,\partial\rho, \partial^2 \partial_s \phi; \partial^3 \psi_0)
\end{align}
where
$
\lambda_1 := \frac{ \partial_{\psi_0} L_0 \psi_0}{ |\nabla \psi_0|^2}$ and $\lambda_2 :=  \frac{\partial_{s} L_0 \psi_0}{ |\nabla \psi_0|^2}$.
To get the desired equation for $\partial_s \phi$ we must use the equations that $\psi_0$ and $\psi$ satisfy.  Recall
\be
  L\psi = F(\psi) + G(x ,\psi), \qquad  L_0\psi =   F_0(\psi_0) + G_0(y ,\psi_0),
\ee
Then we have
\begin{align}
 (  L \psi )\circ \gamma -  F(\psi_0)+ G(\gamma ,\psi_0) = 0,
\end{align}
Thus upon substitution we obtain
\begin{align}
  (L_0 + \Lambda) \partial_s \phi  &=L_0 \psi_0 -\rho^2 (  L \psi )\circ \gamma   + \Lambda_2 \partial_{\psi_0} \phi +\sL_5+ \sN_5\\
  &=  F_0(\psi_0) -  F(\psi_0) + G_0(y ,\psi_0)- G(\gamma ,\psi_0)\\
  &\qquad    + \Lambda_2 \partial_{\psi_0} \phi +\sL_5+ \sN_5.
\end{align}
Introducing the notation
\begin{align}
 \sR_{G}(\partial \phi, \partial \eta, \partial^2_Y G) &:=G_0(y ,\psi_0)- G(\gamma ,\psi_0)- (\gamma-y) \cdot (\nabla_y G) \circ \gamma,
\end{align}
we further express the nonlinear  $G$ terms as follows
\begin{align}
G_0(y ,\psi_0)- G(\gamma ,\psi_0)&=   \frac{1}{|\nabla \psi_0|^2} \Big(( \partial_s \phi) (\partial_{\psi_0} G)(\gamma, \psi_0) - (\partial_{\psi_0} \phi) (\partial_s  G)(\gamma, \psi_0)\Big) \\
&\qquad- \nabla \eta \cdot (\nabla_y G)(\gamma, \psi_0)  + (G-G_0) (\gamma,\psi_0) +  \sR_{G}.\label{nonlinG}
\end{align}
Together, we obtain
\begin{align}
  (L_0 - \Lambda) \partial_s \phi  &= \Lambda_2 \partial_{\psi_0} \phi +(F-F_0)(\psi_0)  + \sL_\phi(\delta a, \delta b,\partial a_0, \partial b_0, \partial^3 \eta, \partial\rho, \partial^2 \partial_s\phi; \partial^3 \psi_0)\\
  &\qquad + \sN(\partial^2 a_0, \partial^2 b_0, \partial^2\phi, \partial^2 \eta,\partial\rho, \partial^2 \partial_s \phi; \partial^3 \psi_0)
\end{align}
where we have defined
\begin{align}\label{Lphi}
\sL_\phi&=\sL_5 ,\\
 \sN_\phi&= \sN_5+ \eqref{nonlinG} \label{Nphi}
\end{align}
where $\sL_5$ and $\sN_5$ are defined in \eqref{LN5} and where
\begin{align}
\Lambda_1 &:= \frac{1}{ |\nabla \psi_0|^2}\Big( \partial_{\psi_0} L_0 \psi_0-  \partial_{\psi_0} G(y,\psi_0) \Big), \qquad \Lambda_2 :=  \frac{1}{ |\nabla \psi_0|^2}\Big( \partial_{s} L_0 \psi_0 -  \partial_{s} G(y,\psi_0) \Big).
\end{align}
Note that we separate out $(F-F_0)(\psi_0)  $ since we will use $F$ to fix $ \partial_s \phi$ as mean-zero on streamlines during the construction.
We note now that
\begin{align}
\partial_{\psi_0} L_0 \psi_0 &=  |\nabla \psi_0|^2  F'_0(\psi_0) + \partial_{\psi_0} G + |\nabla \psi_0|^2 G_0'(y,\psi_0), \qquad \partial_{s} L_0 \psi_0 =  \partial_{s} G_0,
\end{align}
where $ G_0'$ denotes differentiation with respect to its $\psi_0$ argument. Thus,  introducing
\begin{align}\label{lambda}
\Lambda &:=F_0'(\psi_0) + G_0'(y,\psi_0) ,
\end{align}
we obtain
\begin{lemma} \label{philem} If $\psi_0$ solves  \eqref{modelpbm} and $\psi= \psi_0\circ \gamma$  solves \eqref{modelpbmdef} then $\partial_s \phi $ satisfies
\begin{align}\label{etaeq}
\Delta \eta&= \rho-1 +   \sN_\eta(\partial^2 \eta, \partial^2\phi),\\
  (L_0 - \Lambda) \partial_s \phi  &=(F-F_0)(\psi_0) + \sL_\phi(\delta a, \delta b,\delta F,  \partial\delta G, \partial a_0, \partial b_0, \partial^3 \eta, \partial\rho, \partial^2 \partial_s\phi, \partial_{\psi_0} \phi; \partial^3 \psi_0)\\
  &\qquad \qquad + \sN_\phi (\partial^2 a_0, \partial^2 b_0, \partial^2\phi, \partial^2 \eta,\partial\rho, \partial^2 \partial_s \phi; \partial^3 \psi_0), \label{finaleqn}
\end{align}
where $\Lambda$ is given by \eqref{lambda}, $\sL_\phi$ (defined by \eqref{Lphi}) are all the collected terms which are linear in $\phi$ and $\eta$ (and their derivatives), but all multiplied by small factors, $\sN_\eta$ is defined by \eqref{Neta} and $\sN_\phi$ collects the nonlinear terms above (defined by \eqref{Nphi}) .
\end{lemma}

\section{Proof of Theorem \ref{thm1}}

\subsection{Perturbative Assumptions}\label{smallness}

We will make the following assumptions that ensure that various quantities
we will encounter can be treated perturbatively.
\begin{itemize}
 \item The density $\rho$ satisfies
 \begin{equation}
  \|\rho-1\|_{C^{k,\alpha}(D_0)} \leq \epsilon_1.
  \label{assump1}
 \end{equation}

 \item The boundary $\pa D$ is given by $\{B = 0\}$
 and $\pa D_0$ is given by $\{B_0 = 0\}$ where $B, B_0$ are smooth
 functions defined in a neighborhood of $\pa D_0$ and
 \begin{equation}
  \|B - B_0\|_{C^{k,\alpha}} \leq \epsilon_2.
  \label{assump2}
 \end{equation}

 \item The operators $L_0, L$ are close in the sense that the coefficients
 satisfy
 \begin{equation}
  \|a - a_0\|_{C^{k,\alpha}(D_0)} + \|b - b_0\|_{C^{k,\alpha}(D_0)}
  \leq \epsilon_3.
  \label{assump3}
 \end{equation}
  \item The nonlinearities/forcings are close in the sense that
 \begin{equation}
 \|G - G_0\|_{C^{k-2,\alpha}(D_0)}
  \leq \epsilon_4.
  \label{assump4}
 \end{equation}
\end{itemize}
The size of the parameters $\epsilon_1, \epsilon_2, \epsilon_3,\epsilon_4$ will be
set in Lemma \ref{uniformbds} and depends on $D_0$, the base solution $\psi_0$,
and the operator $L_0$.

\subsection{Boundary conditions}
Suppose $D_0$ is given as the interior of a Jordan curve $B_0$ in $\R^2,$
\begin{equation}
 \pa D_0 = \{ p \in \R^2 \  |\  B_0(p) = 0\}.
 \label{}
\end{equation}
For convenience we will assume, without loss of generality, that $B_0$ is given so that $|\nabla B_0| = 1$
and $\nabla \psi_0\cdot \nabla B_0 > 0$.
Suppose that $D$ is given as the interior of a Jordan curve $B$,
\begin{equation}
 \pa D = \{p\in \R^2\ | \ B(p) = 0\}.
 \label{}
\end{equation}
If $\gamma:D_0 \to D$ is of the form
$\gamma = {\rm id} + (\alpha, \beta)$, then
using that $B_0|_{\pa D_0} = 0$, the requirement that $\gamma:\pa D_0 \to\pa D$
can be written as
\begin{align}
 0 = B\circ \gamma|_{\partial D_0}  &= B_0\circ \gamma|_{\partial D_0} + ( \delta B)\circ \gamma|_{\partial D_0} \\
 &=
\alpha \pa_1 B_0|_{\partial D_0} + \beta \pa_2 B_0|_{\partial D_0}  + B_1(\alpha, \beta )|_{\partial D_0}
 \label{bc}
\end{align}
where the remainder $B_1$ is
\begin{equation}
 B_1(\alpha, \beta, x,y)
 =  B_0\circ \gamma - B_0   -\alpha \pa_1 B_0 -\beta \pa_2 B_0 + ( \delta B)\circ \gamma
 \label{B1def}
\end{equation}
which will be small, $O(\alpha^2, \beta^2, \delta B)$, provided $\alpha, \beta, \delta B$ are small and that $B_0 \in C^2$,
say. It is convent to write $\alpha, \beta$ in terms of a gradient and skew gradient of $\phi, \eta$,
\be
(\alpha,\beta) =  \nabla^\perp \psi + \nabla \eta.
\ee
In this case,
\begin{equation}
\alpha \pa_1 B_0 + \beta \pa_2 B_0
 =  \nabla^\perp B_0\cdot \nabla \phi +  \nabla B_0\cdot \nabla \eta.
 \label{}
\end{equation}
Since $|\nabla B_0|=1$, to follows that $ \nabla B_0$ is the outward-facing unit normal vector field, $\hat{n}$ to $D_0$
and $ \nabla^\perp B_0$ is the unit tangential vector field forming
a right-handed basis with  $ \nabla B_0$. Since $\psi_0$ is constant on $\pa B_0$ we in
fact have
$
 \nabla^\perp B_0 = \frac{\nabla^\perp \psi_0}{|\nabla \psi_0|} .$
Using this, we re-write \eqref{bc} as the condition
\begin{equation}
 \frac{1}{|\nabla \psi_0|} \pa_s \phi + \partial_{{n}}  \eta = -B_1(\phi, \eta), \qquad \text{on} \qquad \partial D_0.
 \label{}
\end{equation}
We will choose $\eta$ so that $\partial_{{n}} \eta$ is constant on the boundary
and so that $\pa_s \phi$ has zero average along streamlines, i.e. $\oint_{\psi_0} \pa_s \phi\  \rmd s=0$ where $\rmd s = \rmd \ell /|\nabla \psi_0|$ and $\ell$ is the arc-length parameter.
We will construct $\eta, \phi$ so that they satisfy
\begin{alignat}{2}
\partial_{{n}}  \eta &= -\frac{ \oint_{\partial D_0}
   B_1(\phi, \eta)\, {d\ell}}{{\rm length}(\partial D_0)}
 &&\quad \text{ on } \pa D_0,
 \label{etabc}\\
 \pa_s \phi &= |\nabla \psi_0| \left(-B_1(\phi, \eta) +\frac{ \oint_{\partial D_0}
   B_1(\phi, \eta)\, {d\ell}}{{\rm length}(\partial D_0)} \right)
 &&\quad \text{ on } \pa D_0.
 \label{phibc}
\end{alignat}
This choice is made so that the integral of the right-hand side of \eqref{phibc}
along streamlines is zero.

\subsection{ Governing equations for $\eta$ and $\pa_s \phi$}
By Proposition \ref{pset}, if $\det \nabla \gamma = \rho$, we have
\begin{alignat}{2}
 \Delta \eta &= \rho -1 + \sN_\eta, && \quad \text{ in } D_0,
 \label{etasys1}\\
\partial_n \eta &= - \frac{ \oint_{\partial D_0}
   B_1(\phi, \eta)\, {d\ell}}{{\rm length}(\partial D_0)} &&\quad \text{ on } \pa D_0,
 \label{}
\end{alignat}
where $B_1$ is defined as in \eqref{B1def}, and
where $ \sN_\eta$ is a homogeneous quadratic polynomial depending on $\pa^2 \eta, \pa^2 \phi$ defined in \eqref{Neta}.
The equation for $\pa_s \phi$ takes the form
\begin{alignat}{2}
\big( L_0-\Lambda\big) \pa_s\phi &= (F- F_0)  + \sL_\phi+ \sN_\phi, &&\qquad \text{ in } D_0,\label{phisys1}\\
 \pa_s \phi &=|\nabla \psi_0|\left(- B_1(\phi, \eta) +
\frac{ \oint_{\partial D_0}
   B_1(\phi, \eta)\, {d\ell}}{{\rm length}(\partial D_0)} \right)
 &&\qquad \text{ on } \pa D_0,
 \label{phisys2}
\end{alignat}
where $ L_0 $ is the elliptic operator defined in \eqref{modelpbm} and
 $\sL_\phi$ are terms which are linear in derivatives of $\phi$ and $\eta$ and $\rho$ multiplied by small factors and $\sN_\phi$ are quadratically nonlinear terms in derivatives of $\phi$ and $\eta$ and $\rho$.
 In this formulation, the function $F$ is unknown and will need to be chosen to
 be consistent with the fact that $\oint \pa_s\phi d \ell= 0$. This point
 will be explained in detail in the next section.
We emphasize that $\sL$ and $\sN$ do not involve arbitrary third derivatives of $\eta, \phi$
and it is only $ \pa^2 \pa_s\eta,  \pa^2 \pa_s\phi$ that enter, which will be
important in what follows. The following estimates are immediate consequences
of the definitions of the terms on the right-hand sides of
\eqref{etasys1}-\eqref{phisys1} in which can be found in \eqref{Neta}, \eqref{Lphi} and \eqref{Nphi}.
Note that these quantities involve three derivatives of $\psi_0$ but we are
assuming $F_0 \in C^{k-1,\alpha}$ so by standard elliptic estimates
$\|\psi_0\|_{C^{k+1,\alpha}}$ is finite.

\begin{lemma}\label{bndsNL}
If the bounds in \S \ref{smallness} hold, then we have
\begin{align}
 \|\sL_\phi\|_{C^{k-2,\alpha}(D_0)}& \leq C_{k,\alpha} \big( \|\rho -1\|_{C^{k-1,\alpha}(D_0)}
 + \|\eta\|_{C^{k,\alpha}(D_0)} + \|\pa_s \eta\|_{C^{k,\alpha}(D_0)}
 + \epsilon \|\pa_s\phi\|_{C^{k-1,\alpha}(D_0)}\big),\\
 \|\sN_\phi\|_{C^{k-2,\alpha}(D_0)}& \leq C_{k,\alpha} \big( \|\rho -1\|_{C^{k-1,\alpha}(D_0)}
 + (\|\eta\|_{C^{k,\alpha}(D_0)} + \|\pa_s \eta\|_{C^{k,\alpha}(D_0)}
 + \|\pa_s \phi\|_{C^{k,\alpha}(D_0)} )^2\big),\\
  \|\sN_\eta\|_{C^{k-2,\alpha}(D_0)}& \leq C_{k,\alpha} \big( \|\eta\|_{C^{k,\alpha}(D_0)}
 + \| \phi\|_{C^{k,\alpha}(D_0)} )^2,\\
 \|B_1\|_{C^{k-1,\alpha}(\pa D_0)} &\leq C_{k,\alpha}\Big( \|\eta\|_{C^{k,\alpha}(D_0)}
   + \|\phi\|_{C^{k,\alpha}(D_0)}\Big),
\end{align}
where $\epsilon =\max\{ \epsilon_1,\epsilon_2,\epsilon_3,\epsilon_4\}$.
\end{lemma}

We also need Lipschitz bounds for the operators $\sL_\phi, \sN_\phi, \sN_\eta$.
Given functions $\phi_1, \eta_1, \phi_{1}, \eta_{2}$
we write $ u_1 = (\phi_1, \eta_1), u_{2} = (\phi_{2}, \eta_{2})$ and
let $\sL_{\phi}^i, \sN_\phi^i, \sN_\eta^i$ for $i = 1,2$ denote the operators
$\sL_\phi, \sN_\phi, \sN_\eta$ defined in Proposition \ref{pset} evaluated
at $(\phi_i, \eta_i)$. The following estimates are then straightforward consequences
of the definitions.
\begin{lemma}\label{diffbndsNL}
 If the bounds in \S \ref{smallness} hold, then we have
 \begin{align}
  \|\sL_\phi^1 - \sL_\phi^2 \|_{C^{k-2,\alpha}(D_0)}& \leq C_{k,\alpha}
   \big(
    \|\eta_1 - \eta_2\|_{C^{k,\alpha}(D_0)} + \|\pa_s \eta_1 - \pa_s \eta_2\|_{C^{k,\alpha}(D_0)}
   + \epsilon \|\pa_s\phi_1 - \pa_s \phi_2\|_{C^{k,\alpha}(D_0)}\big),\\
  \|\sN_\phi^1 - \sN_\phi^2\|_{C^{k-2,\alpha}(D_0)}&
   \leq C_{k,\alpha} \big(
   \|\eta_1 - \eta_2\|_{C^{k,\alpha}(D_0)} + \|\pa_s \eta_1 -\pa_s \eta_2\|_{C^{k,\alpha}(D_0)}
  + \|\pa_s \phi_1 - \pa_s \phi_2\|_{C^{k,\alpha}(D_0)} )^2,\\
   \|\sN_\eta^1 - \sN_\eta^2\|_{C^{k-2,\alpha}(D_0)}& \leq C_{k,\alpha} \big( \|\eta_1 - \eta_2\|_{C^{k,\alpha}(D_0)}
  + \| \phi_1 - \phi_2\|_{C^{k,\alpha}(D_0)} )^2.
 \end{align}
\end{lemma}
\vspace{1mm}

\subsection{ Recovering $\phi$ from  $\pa_s \phi$}

In the construction, a solution $\Phi$  is obtained by solving \eqref{phisys1}--\eqref{phisys2} for ``$\partial_s \phi = \nabla^\perp \psi_0 \cdot \nabla \psi$".  Consistent with $\Phi=\partial_s \phi$, we will construct a solution  $\Phi$ with the property that its integral on each streamline is zero.  This is done further in the proof and requires the use of {\rm (\textbf{H3})}.  To verify that it is indeed the ``streamline derivative" and to recover the periodic function $\phi$, we appeal to the following lemma.

\begin{lemma}
 \label{zeroavg}
Suppose $\Phi\in C^{k,\alpha} $ satisfies
 \begin{equation}
\oint_{\{\psi_0= c\}}  \Phi  = 0 \qquad \text{for all} \qquad c\in {\rm im}(\psi_0).
  \label{}
 \end{equation}
 Then $\Phi = \pa_s  \phi$ for a unique function $\phi = \phi(\psi_0,\theta)$ which is a zero-mean, periodic function on streamlines
 of $\psi_0$, i.e. $\phi(\psi_0, 0) = \phi(\psi_0, 2\pi)$ and $\oint_{\psi_0} \phi=0$.  Moreover, $\phi$ enjoys the bound
 \be\label{phibd}
\|\phi\|_{k,\alpha} \leq C \|\pa_s \phi\|_{k,\alpha},
\ee
where the constant $C$ depends on  $C^{k,\alpha}$ norms of $\psi_0$.
\end{lemma}
\begin{proof}
First  note that all that enters in the formulation of the problem
is $\pa_s \phi$ and not $\phi$ itself and so we are free to modify $\phi$
by adding an arbitrary function of $\psi_0$. To be more precise, let $  \ell$ be the arc-length along curves $\{\psi=c\}$ and introduce the notation
\be
\rmd s =  \frac{\rmd \ell}{|\nabla \psi_0|}.
\ee
We the fix the freedom in defining $\phi$ by enforcing that, on each streamline,
\begin{equation}
 \oint_{\{\psi_0=c\}} \phi \ \rmd s =0, \qquad \forall c\in {\rm im}(\psi_0).
 \label{meanzerophi}
\end{equation}
Assuming that this holds, we have
$\|\phi\|_{k,\alpha} \leq C \|\pa_s \phi\|_{k,\alpha}$, a fact that we use
repeatedly in what follows.
To be more precise, we introduce an ``angular coordinate" along streamlines as
\be
\theta(x)= \frac{2\pi }{\mu(\psi_0(x))}  \int_{\Gamma_{x_0(\psi_0),x} } \rmd s , \qquad \mu(c) = \oint_{\{\psi_0=c\}} \rmd s
\ee
where $\mu$ is the travel time of a particle along a streamline and where, for each $x\in D_0$ the line integral is taken counterclockwise from an arbitrary point $x_0(\psi_0)$ on the streamline to the point $x$. This  point $x_0(\psi_0)$ can be obtain by flowing an arbitrary point $p\in D_0$ by the vector field $\nabla \psi_0$ which is orthogonal to streamlines. This segment is denoted by $\Gamma_{x_0(\psi_0),x}$.
 Then $\theta(x)$ is a $2\pi$--periodic parametrization of the streamline with value $\psi_0(x)$.

Now, given a $\Phi$ which is mean zero on streamlines, note that for an arbitrary $\theta_0\in[0,2\pi]$
\begin{equation}
 \phi(\psi_0, \theta) = \phi(\psi_0, \theta_0) + \mu(\psi_0) \int_{\theta_0}^{\theta} \Phi \,
\rmd \theta' . \qquad
 \label{}
\end{equation}
  Integrating this expression along the streamline in $s_0$ we
$\psi_0 = $const.
\be
 \phi(y) = \phi(\psi_0(y), \theta(y)), \qquad \phi(\psi_0, \theta)  = \mu(\psi_0) \int_\mathbb{T}  \left(  \int_{\theta_0}^{\theta}\Phi  \, \rmd \theta' \right) \rmd \theta_0.
\ee
 One can check that $\nabla^\perp\psi_0\cdot \nabla \theta=\mu^{-1}$.
Thus, for the quantity defined above we have that
\be
\Phi =\partial_s \phi.
\ee
The definitions of $\theta$ and $\mu$ as functions on $D_0$ we obtain the estimate \eqref{phibd}.
\end{proof}

\subsection{The iteration to solve the nonlinear elliptic system}

We use the following iteration. Given $\eta^n, \phi^n
\in C^{k-1, \alpha}(D_0)$ with
$\pa_s \eta^n, \pa_s \phi^n \in C^{k-1,\alpha}(D_0)$ and
$\oint_{\psi_0} \phi^n = 0$, set
\begin{align}
\sN_\eta^n := \sN_\eta(\eta^n, \phi^n),
 \label{}
\end{align}
with $\sN_\eta$ defined in \eqref{Neta}, which satisfies the following bound
\begin{equation}
 \|\sN_\eta^n\|_{C^{k-2,\alpha}(D_0)} \leq C_{k,\alpha} \big(\|\phi^n\|_{C^{k,\alpha}(D_0)} +
  \|\eta^n\|_{C^{k,\alpha}(D_0)}\big)^2.
 \label{}
\end{equation}
By Lemma \ref{neumann} the following problem has a unique solution
$\eta^{n+1} \in C^{k-1,\alpha}$,
\begin{alignat}{2}
 \Delta \eta^{n+1} &= \rho - 1 + \sN_\eta^n, &&\qquad \text{ in } D_0,\label{etan1eqn}\\
\partial_n \eta^{n+1} &= \kappa^{n+1} \equiv \int_{D_0}\Big( \rho - 1 + \sN_\eta^n\Big)\,
 &&\qquad \text{ on } \pa D_0.
 \label{etanbc}
\end{alignat}
Moreover, the iterate $\eta^{n+1}$ enjoys the following estimate
\begin{align}
 \|\eta^{n+1}\|_{C^{k,\alpha}(D_0)}
 \leq C_{k,\alpha} \big( \|\rho - 1\|_{C^{k-2,\alpha}(D_0)}
 + \|\sN_\eta^n\|_{C^{k-2,\alpha}(D_0)}
 + \kappa^{n+1}\big).
 \label{}
\end{align}
We note that \eqref{etanbc} does not agree with \eqref{etabc} but instead
has been chosen to ensure that the Neumann problem is solvable. In the
upcoming Lemma \ref{areaconvergence} we show that provided $\eta^n, \phi^n$
converge, the limit $\eta$ will satisfy \eqref{etabc} as a consequence of
the assumption that ${\rm Vol}(D) = \int_{D_0} \rho $.

In order to get an estimate for $\|\pa_s \eta^{n+1}\|_{C^{k,\alpha}(D_0)}$
we commute the equation \eqref{etan1eqn}-\eqref{etanbc} with $\pa_s$.
Applying $\pa_s$ to \eqref{etan1eqn}, using
\be
[\partial_s,\Delta]=- 2 \nabla^\perp \Delta \psi_0 \cdot \nabla  - \nabla \otimes \nabla^\perp \psi_0 : \nabla \otimes \nabla ,
\ee
we note
that the right-hand side involves highest-order derivatives falling on
$\pa_s \eta$ and lower-order terms.
By the estimates for the Neumann problem from Lemma \ref{neumann}
and using that $\pa_s \eta = 0$ on the boundary since $\pa_s$ is a tangential derivative,
we have
\begin{equation}
 \|\pa_s \eta^{n+1}\|_{C^{k,\alpha}(D_0)} \leq C_{k,\alpha}  \| \eta^{n+1}\|_{C^{k,\alpha}(D_0)} + C_{k,\alpha} \big( \|\pa_s(\rho - 1)\|_{C^{k,\alpha}(D_0)}
 + \|\pa_s \sN_\eta^n\|_{C^{k-2,\alpha}(D_0)} \big).
\end{equation}
With $\eta^{n+1}$ defined, we now set
\begin{align}
\sN_\phi^n = \sN_\phi(\eta^{n+1}, \phi^n),\\
 B_1^n = B_1(\eta^{n+1}, \phi^n),
 \label{}
\end{align}
with $\sN_\phi$ defined in \eqref{phisys1} and $B_1$ defined in \eqref{B1def}.
Using that $\|\phi^n\|_{C^{k-1,\alpha}(D_0)} \leq \|\pa_s \phi^n\|_{C^{k-1,\alpha}(D_0)}$ from Lemma \ref{zeroavg}, we have that
\begin{multline}
  \|\sN_\phi^n\|_{C^{k-2,\alpha}(D_0)}
  \leq C_{k,\alpha} \big( \|\rho -1\|_{C^{k,\alpha}(D_0)}
  + (\|\eta^{n+1}\|_{C^{k,\alpha}(D_0)} + \|\pa_s \eta^{n+1}\|_{C^{k,\alpha}(D_0)}
   + \|\pa_s \phi^n\|_{C^{k,\alpha}(D_0)} )^2
  \\
  + \|\eta^{n+1} \|_{C^{k,\alpha}(D_0)} + \|\pa_s \eta^{n+1}\|_{C^{k,\alpha}(D_0)}
  +\epsilon \|\pa_s \phi^n\|_{C^{k-1,\alpha}(D_0)}
  \big),
 \label{}
\end{multline}
with $\epsilon=\max\{\epsilon_1,\epsilon_2,\epsilon_3,\epsilon_4\}$.
We note that it is crucial that the estimate for the nonlinearity $\sN_\phi^n$ only
requires a bound for $\|\pa_s \eta^{n+1}\|_{C^{k,\alpha}(D_0)}$ and not
the full norm $\|\eta^{n+1}\|_{C^{k+1,\alpha}(D_0)}$ since we could only get a
bound for this term by differentiating the equation for $\eta^{n+1}$ in all
directions and this would require a bound for $\|\phi^{n+1}\|_{C^{k+1,\alpha}(D_0)}$
instead of just $\|\pa_s \phi^{n+1}\|_{C^{k,\alpha}(D_0)}$.
The boundary operator satisfies the estimate
\begin{align}
 \|B_1^n\|_{C^{k,\alpha}(D_0)} &\leq C_{k,\alpha}\Big(  \|\eta^{n+1} \|_{C^{k,\alpha}(D_0)}  + \| \delta B \|_{C^{k,\alpha}(D_0)} \Big).
 \label{}
\end{align}

We now envoke hypothesis (\textbf{H3}) and define $F^n$ by the requirement that the right-hand side of
\eqref{phieqnprop} has zero average along streamlines with $\phi, \eta$ replaced
by $\phi^n, \eta^n$.  Consider the problem
\begin{align}
(L_0 - \Lambda) u &=  g \qquad\ \  \text{in }  D_0\\
u&=u_b \qquad \text{on } \partial D_0.
\end{align}
Letting $G$ be the Green's function for the Dirichlet problem for $L_0 - \Lambda$, we have
\be
 u(x) = \int_{D_0} G(x,x') g(x')\rmd x' +  \oint_{\partial D_0} \partial_n G(x,x') u_b(x') \rmd  \ell.
\ee
We define
\be
( L_0 - \Lambda)^{-1}_{\rm hbc}  g := \int_{D_0} G(x,x') g(x')\rmd x',
\ee
so that $f = (L_0 - \Lambda)^{-1}_{\rm hbc}g$ means $(L_0 - \Lambda) f = g$
and $f = 0$ on $\pa D_0$.
If the equations \eqref{phisys1}-\eqref{phisys2} are to hold then since
$\oint_{\psi_0 = c} \pa_s\phi = 0$ we must ensure that
 \begin{align}\label{Feqn}
K_{\psi_0}[ F^n- F_0] = -\oint_{\psi_0} ( L_0 - \Lambda)^{-1}_{\rm hbc} [\sL_\phi^n  + \sN_\phi^n ] \rmd s  + \oint_{\psi_0} \oint_{\partial D_0} \partial_n G(x,x') \partial_s \phi(x') \rmd  \ell\rmd s,
\end{align}
with $K_{\psi_0}$ defined in the statement of (\textbf{H3}).
Notice that the right-hand-side is a function only $\psi_0$, and the boundary conditions for $\partial_s\phi$ are chosen exactly so that the contribution on the boundary of the final term in \eqref{Feqn} is zero.  Thus,  with $g^n(\psi_0)$ defined as the right-hand-side of equation \eqref{Feqn}, we see that $g^n(\psi_0(\pa D_0))=0$.  Moreover, $g\in C^{k,\alpha}(I)$ with $I={\rm im}(\psi_0)$, which follows from  Lemma \ref{geolem} Appendix \ref{stream}.
These verify that we are in the setting of (\textbf{H3}) and so by assumption
there is an
$F^n = F^n(\psi_0)\in C^{k-2,\alpha}(I)$ which ensures that \eqref{Feqn} holds.  Moreover,
\be\label{F0Fbd}
\|F^n-F_0\|_{C^{k-2,\alpha}(I)} \lesssim\| K_{\psi_0} [\sL_\phi^n  + \sN_\phi^n ] \|_{C^{k,\alpha}}\lesssim \| ( L_0 - \Lambda)^{-1}_{\rm hbc}\Big(\sL_\phi^n  + \sN_\phi^n\Big)\|_{C^{k,\alpha}} \lesssim \| \sL_\phi^n  + \sN_\phi^n\|_{C^{k-2,\alpha}(D_0)}.
\ee
The second inequality above follows from Lemma \ref{Feqn}.

By Lemma \ref{dirichlet} the following problem has a unique solution
$ \Phi^{n+1}  \in C^{k,\alpha}$,
\begin{alignat}{2}
 (L_0 - \Lambda) \Phi^{n+1} &= (F^n-F_0)(\psi_0)   + \sL_\phi^n  + \sN_\phi^n,&&\qquad \text{ in } D_0, \label{phineqn}\\
 \Phi^{n+1} &=
|\nabla \psi_0|\left(- B_1(\phi^n, \eta^n) +
\frac{ \oint_{\partial D_0}
   B_1(\phi^n, \eta^n)\, {d\ell}}{{\rm length}(\partial D_0)} \right) 
 &&\quad \ \ \text{ on } \pa D_0.
 \label{phinbc}
\end{alignat}
By our choice  for $F^n$ and the above discussion, the solution $  \Phi^{n+1} $
has zero average along streamlines and so by Lemma \ref{zeroavg} it follows
that $ \Phi^{n+1}  = \pa_s \phi^{n+1}$ for a unique function $\phi^{n+1}$ with
zero average along streamlines. From \eqref{phineqn}-\eqref{phinbc} and \eqref{F0Fbd} we have
\begin{multline}
 \|\pa_s \phi^{n+1}\|_{C^{k,\alpha}(D_0)}
 + \|\phi^{n+1}\|_{C^{k,\alpha}(D_0)}
 \\
 \leq C_{k,\alpha}
 \big( \|\sL_\phi^n\|_{C^{k-2,\alpha}(D_0)} + \| \sN_\phi^n\|_{C^{k-2,\alpha}(D_0)} + \|B_1^n\|_{C^{k,\alpha}(\pa D_0)}\big).
 \label{}
\end{multline}

In summary, using Lemma \ref{bndsNL}, we have shown
\begin{lemma}
  \label{iteratedef}
  Suppose that {\rm (\textbf{H1})--(\textbf{H3})} and the assumptions \eqref{assump1}, \eqref{assump2}, \eqref{assump3} and \eqref{assump4} hold.
  Let $\phi^n, \eta^n \in C^{k,\alpha}(D_0)$ with
  $\pa_s \phi^{n}, \pa_s \eta^n \in C^{k,\alpha}(D_0)$ be given functions. With
  $\sN_\eta^n, \sN_\phi^n, B_1^n$ defined as above,
  and with $F^{n}$ defined implicitly by \eqref{Feqn},
  the problems \eqref{etan1eqn}-\eqref{etanbc} and  \eqref{phineqn}-\eqref{phinbc}
  have a unique solution $\eta^{n+1}, \phi^{n+1}$ satisfying
  $\oint_{\psi_0} \phi^{n+1}\, d\ell = 0$, and we have the estimates
  \begin{align} \nonumber
   \|\pa_s \eta^{n+1}\|_{C^{k,\alpha}} + \|\eta^{n+1}\|_{C^{k,\alpha}}&\leq
   C_{k,\alpha} \big(\|\pa_s (\rho - 1)\|_{C^{k,\alpha}}
   + \|\rho - 1\|_{C^{k-2,\alpha}} \\
    &\qquad +
   (\|\pa_s \eta^{n}\|_{C^{k,\alpha}} + \|\eta^{n}\|_{C^{k,\alpha}})
    \|\pa_s \phi^{n}\|_{C^{k,\alpha}}
    ,
    \label{etabds}\\ \nonumber
   \|\pa_s \phi^{n+1}\|_{C^{k,\alpha}} +
   \|\phi^{n+1}\|_{C^{k,\alpha}} &\leq
   C_{k,\alpha} \big( \|\rho-1\|_{C^{k,\alpha}}
   +   (\|\pa_s \eta^{n+1}\|_{C^{k,\alpha}} + \|\eta^{n+1}\|_{C^{k,\alpha}})
   \|\pa_s \phi^{n}\|_{C^{k-1,\alpha}}\\
          &\qquad + (\|\pa_s \eta^{n+1}\|_{C^{k,\alpha}} + \|\eta^{n+1}\|_{C^{k,\alpha}})\\
          &\qquad\qquad
       + \epsilon \|\pa_s \phi^{n}\|_{C^{k,\alpha}} + \|B_1^n\|_{C^{k,\alpha}(\pa D_0)},
   \label{phibds}\\
   \|B_1^n\|_{C^{k,\alpha}(\pa D_0)} &\leq C_{k,\alpha}\Big( \|\eta^{n}\|_{C^{k,\alpha}}
   + \|\phi^n\|_{C^{k,\alpha}}\Big).    \label{B1bdd}
 \end{align}
\end{lemma}

\subsection{Uniform estimates for the iterates}
We now set $\eta^0 = \phi^0 = 0$. Given $\eta^\ell, \phi^\ell$, using
Lemma \ref{iteratedef} let $\eta^{\ell+1}$ satisfy \eqref{etan1eqn}-\eqref{etanbc}
and let $\Phi^{\ell+1} = \pa_s\phi^{\ell+1}$ satisfy \eqref{phineqn}-\eqref{phinbc}. In this section
we prove that the sequences $(\eta^\ell,\phi^\ell), (\pa_s \eta^\ell, \pa_s \phi^\ell)$
 are uniformly bounded in $C^{k,\alpha}(D_0)$.
\begin{lemma}
  \label{uniformbds}
  There $\epsilon_0 = \epsilon_0(D_0, k, \alpha, \theta) > 0$ so that
  if the assumptions \eqref{assump1}-\eqref{assump3} hold with $\epsilon_1 + \epsilon_2 + \epsilon_3
  \leq \epsilon_0/2$, if the sequence $\phi^\ell, \eta^\ell$ is defined as above, then
   \begin{equation}
     \|\pa_s \eta^\ell\|_{C^{k,\alpha}(D_0)} +
     \|\pa_s \phi^\ell\|_{C^{k,\alpha}(D_0)} +
     \| \eta^\ell\|_{C^{k,\alpha}(D_0)} +
     \| \phi^\ell\|_{C^{k,\alpha}(D_0)}
     \leq 1.
    \label{}
   \end{equation}
\end{lemma}
\begin{proof}
  Let $C_{k,\alpha}$ be as in \eqref{etabds}-\eqref{phibds}
  and set
  \begin{equation}
   \epsilon_0 = \min(1, 1/(4 (C_{k,\alpha} + C_{k,\alpha}^2))).
   \label{}
  \end{equation}
  Let $\epsilon = \epsilon_1 + \epsilon_2 +\epsilon_3$ and set $M = 4C_{k,\alpha} \epsilon$.
  We claim that if
  $\epsilon \leq \epsilon_0/2$ then the iterates $\phi^\ell, \eta^\ell$ satisfy
  \begin{equation}
    \|\eta^\ell \|_{C^{k,\alpha}(D_0)} + \|\phi^\ell\|_{C^{k,\alpha}(D_0)}
    +
    \|\pa_s \eta^\ell\|_{C^{k,\alpha}(D_0)} + \|\pa_s \phi^\ell\|_{C^{k,\alpha}(D_0)}
    \leq M \leq 1.
   \label{}
 \end{equation}
 This certainly holds for $\ell = 0$. If it holds for $\ell = 0,..., m-1$
 then by \eqref{etabds}-\eqref{phibds} we have
  \begin{equation}
   \|\eta^m\|_{C^{k,\alpha}(D_0)} + \|\phi^m\|_{C^{k,\alpha}(D_0)}
   +
   \|\pa_s \eta^m\|_{C^{k,\alpha}(D_0)} + \|\pa_s \phi^m\|_{C^{k,\alpha}(D_0)}
   \leq C_{k,\alpha} ( M^2 + \epsilon M + \epsilon)
   \leq M,
   \label{}
  \end{equation}
  since $C_{k,\alpha} M^2 \leq \frac{1}{2}M,$ $C_{k,\alpha} \epsilon M \leq
  \frac{1}{4}M^2 \leq \frac{1}{4} M$ and $C_{k,\alpha} \epsilon \leq \frac{1}{4} M$
  if $\epsilon \leq \epsilon_0/2$. The result follows.
\end{proof}

\subsection{Cauchy estimates for the iterates}

\begin{lemma}
  \label{cauchybds}
  There is $\epsilon'_0 = \epsilon'_0(D_0, k, \alpha, \theta) > 0$ with
  the following property.
  If the assumptions \eqref{assump1}-\eqref{assump3} hold with $\epsilon_1 + \epsilon_2 + \epsilon_3
  \leq \epsilon'_0/2$, then with the sequence $\{\phi^\ell, \eta^\ell\}$ defined
  as in the previous lemma, if we set
  \begin{align}
   D_{N, M} &= \|\pa_s \eta^N - \pa_s \eta^M \|_{C^{k-1,\alpha}(D_0)} +
   \|\pa_s \phi^N -\pa_s \phi^M\|_{C^{k-1,\alpha}(D_0)} \\
   &\qquad +
   \| \eta^N - \eta^M\|_{C^{k-1,\alpha}(D_0)} +
   \| \phi^N - \eta^M\|_{C^{k-1,\alpha}(D_0)},
   \label{}
  \end{align}
  then     $D_{N, M} \leq \frac{1}{2} D_{N-1, M-1}.$
   In particular, with $d_1 = \|\pa_s \eta^1\|_{C^{k-1,\alpha}(D_0)} +
   \|\pa_s \phi^1\|_{C^{k-1,\alpha}(D_0)} +
   \| \eta^1\|_{C^{k-1,\alpha}(D_0)} +
   \| \phi^1\|_{C^{k-1,\alpha}(D_0)}$, we have that
   \begin{equation}
    D_{N, M} \leq 2^{1-\min(N, M)} d_1.
    \label{}
   \end{equation}
\end{lemma}
\begin{proof}
 This is proved in nearly the same way as the previous lemma, but relies on
 Lemma \ref{diffbndsNL} in place of Lemma \ref{bndsNL}.
\end{proof}

\subsection{Convergence of the boundary term}

\begin{lemma}
  \label{areaconvergence}
 Let $D_0, D$ be domains in $\R^2$ and suppose that for some function $\rho$,
 \begin{equation}
{ \rm Area }(D)= \int_{D_0} \rho\, \rmd y.
  \label{}
 \end{equation}
 Suppose that $\partial D = \{ B(x) = 0\}$ for some function $B$ defined in a tubular neighborhood of $\partial D$ and has non-vanishing gradients there. Let $\gamma$ be a diffeomorphism of the form $\gamma= {\rm id} + \nabla^\perp \phi
 + \nabla \eta$ where
\begin{alignat}{2}
 \pa_s \phi &= |\nabla \psi_0| \left(-B_1(\phi, \eta) +\frac{ \oint_{\partial D_0}
   B_1(\phi, \eta)\, {d\ell}}{{\rm length}(\partial D_0)} \right),
\qquad  \text{ on } \pa D_0,
\end{alignat}
 with $B_1$ defined as in \eqref{B1def}.
 and where $\partial_n \eta$ is constant on $\pa D_0$. Then in fact
\begin{alignat}{2}
\partial_{{n}}  \eta &= -\frac{ \oint_{\partial D_0}
   B_1(\phi, \eta)\, {d\ell}}{{\rm length}(\partial D_0)},\qquad  \text{ on } \pa D_0,
\end{alignat}
 and as a consequence, $\gamma: \partial D_0 \to \partial D$.
\end{lemma}

\begin{proof} Recall that, by the definition of the map $\gamma$, we have
\begin{align}
 B\circ \gamma|_{\partial D_0}  &=  \frac{1}{|\nabla \psi_0|} \pa_s \phi + \partial_{{n}}  \eta + B_1(\phi, \eta), \qquad \text{on} \qquad \partial D_0.
\end{align}
By the above assumptions, this implies that, for some constant $c$,
\be
 B\circ \gamma|_{\partial D_0} = c.
\ee
  This says that $\gamma$ maps $\partial D_0$ to the level set $\{ B=c\}$.  We wish to conclude that $c=0$ based on the fact that the area of $\gamma(D_0)$ is the same as $D = \overline{\{ B=0\}}$.  Note that the area enclosed  by the level set, $\overline{\{ B=c\}}$, has the property that
\be
\frac{\rmd}{\rmd c} {\rm Area}({\overline{\{ B=c\}}})  = \oint_{{\{ B=c\}}} \frac{\rmd \ell}{|\nabla B|}.
\ee
By our assumptions that $|\nabla B|$ is non-vanishing in a neighborhood of the zero set, then the level sets are  Jordan curves in a neighborhood of $0$ and the area enclosed must change in accord with the above formula.  Thus, the unique value of $c$ such that
\be
 {\rm Area}({\overline{\{ B=c\}}}) =  {\rm Area}({\overline{\{ B=0\}}})
\ee
is $c=0$ and we are done.
\end{proof}

\subsection{Proof of Theorem \ref{thm1}}

By Lemmas \ref{uniformbds} and \ref{cauchybds} it follows that the iterates
$(\eta^\ell, \phi^\ell)$ form a Cauchy sequence in $C^{k+2, \alpha}(D_0)$ and
so they converge to functions $(\eta, \phi) \in C^{k+2, \alpha}(D_0)$
with a corresponding statement for $\pa_s \eta^\ell, \pa_s\phi^\ell$. We then
set $\gamma = {\rm id}+ \nabla \eta + \nabla^\perp \phi$. It remains to show that
$\gamma(D_0) = D$, and by Lemma \ref{areaconvergence} and
 \eqref{etabc}-\eqref{phibc} it follows that $\gamma|_{\pa D_0} = \pa D$
  as required. The estimate \eqref{gammabdthm} follows from the proof of
  Lemma \ref{uniformbds}.

\subsection{Proof of Theorem \ref{thm2}}

We define a sequence of diffeomorphisms $\{\gamma^N\}$ as follows. Given
a domain $D_{N-1}$ and a diffeomorphism
$\gamma^{N-1}: D_0 \to D_{N-1}$ of the form $\gamma^{N-1} ={\rm id}+ \nabla \eta^{N-1}
+ \nabla^\perp \phi^{N-1}$, define $\rho^N$ by
\begin{equation}
 \rho^N(y) =  X(y, \eta^{N-1}, \phi^{N-1}, \nabla \eta^{N-1}, \nabla \phi^{N-1},
  \nabla \pa_s \eta^{N-1}, \nabla \pa_s \phi^{N-1})
 \label{}
\end{equation}
and define
$\sigma^N > 0$ by
$\sigma_{N}^2 = \frac{\int_{D_0} \rho^{N}}{Vol D_0}$ so that with
$D_N = \sigma_N D_0$, we have $Vol D_N = \sigma_{N}^2 Vol D_0 = \int_{D_0} \rho^N$.
By Theorem \ref{thm1} there is a diffeomorphism $\gamma^N: D_0 \to D_N$ with
$\det \nabla \gamma^{N} = \rho^N$ and where $\gamma^{N}$ is of the form
$\gamma^{N} ={\rm id}+ \nabla \eta^N + \nabla^\perp \phi^N$ so that
$\psi^N = \psi_0\circ \gamma^N$ satisfies \eqref{modelpbmdef}, and we
have the estimates
\begin{equation}
 \|\pa_s \eta^N\|_{C^{k-1, \alpha}}
 +
 \|\pa_s \phi^N\|_{C^{k-1, \alpha}}
 +
 \|\eta^N\|_{C^{k-1, \alpha}}
 +
 \| \phi^N\|_{C^{k-1, \alpha}}
 \leq C_{k,\alpha} \big( \|\rho^{N} - 1\|_{C^{k-1,\alpha}} + \ve\big).
 \label{gammaNbdnonlinear}
\end{equation}
Taylor expanding $\rho^N = X(y, \eta^{N-1}, \phi^{N-1}, \nabla \eta^{N-1}, \nabla
\phi^{N-1}, \nabla \pa_s \eta^{N-1}, \nabla \pa_s \phi^{N-1})$ around $(\eta, \phi) = (0, 0)$ and using the bound
\eqref{BXbds} we have
\begin{multline}
 \|\rho^N - 1\|_{C^{k-1, \alpha}}
 \leq C \epsilon_X \big(\|\pa_s\eta^{N-1}\|_{C^{k-1, \alpha}}+ \|\pa_s\phi^{N-1}\|_{C^{k-1, \alpha}}
 +\|\eta^{N-1}\|_{C^{k-1, \alpha}}+ \|\phi^{N-1}\|_{C^{k-1, \alpha}}\big)
 \\
 + C\big(\|\pa_s\eta^{N-1}\|_{C^{k-1, \alpha}}+ \|\pa_s\phi^{N-1}\|_{C^{k-1, \alpha}}
 +\|\eta^{N-1}\|_{C^{k-1, \alpha}}+ \|\phi^{N-1}\|_{C^{k-1, \alpha}}\big)^2,
 \label{gammaNbdnonlinear}
\end{multline}
provided $\|\pa_s\eta^{N-1}\|_{C^{k-1, \alpha}}+ \|\pa_s\phi^{N-1}\|_{C^{k-1, \alpha}}
+\|\eta^{N-1}\|_{C^{k-1, \alpha}}+ \|\phi^{N-1}\|_{C^{k-1, \alpha}} \leq 1$,
say.
We now prove that the sequence $\gamma^{N}$ is uniformly bounded provided
$\ve_X$ is taken sufficiently small.
 Let $C_{k,\alpha}$ be the constant in \eqref{gammabdthm} and take
 $\ve_X$ so small that
$
  4 C_{k,\alpha}C_{k,\alpha}' \ve_X  \leq 1,
$
 and suppose that
 \begin{equation}
   \|\pa_s \eta^N\|_{C^{k-1, \alpha}}
   +
   \|\pa_s \phi^N\|_{C^{k-1, \alpha}}
   +
   \|\eta^N\|_{C^{k-1, \alpha}}
   +
   \| \phi^N\|_{C^{k-1, \alpha}}
   \leq 2 C_{k,\alpha} \ve \leq 1.
  \label{}
 \end{equation}

 By \eqref{gammaNbdnonlinear},  we then have
 \begin{multline}
   \|\pa_s \eta^{N+1}\|_{C^{k-1, \alpha}}
   +
   \|\pa_s \phi^{N+1}\|_{C^{k-1, \alpha}}
   +
   \|\eta^{N+1}\|_{C^{k-1, \alpha}}
   +
   \| \phi^{N+1}\|_{C^{k-1, \alpha}}
   \\
  \leq C_{k,\alpha} \big( C_{k,\alpha}' \ve_X (2\ve)^{k+1+\alpha} + \ve\big)
  \leq 2 C_{k,\alpha} C_{k,\alpha}'  \ve\ve_X  + C_{k,\alpha} \ve
  \leq 2 C_{k,\alpha} \ve,
  \label{}
 \end{multline}
 and it follows that the sequence $\{\gamma^N\}_{N = 0}^\infty$ is uniformly bounded in
 $C^{k+1,\alpha}$. Using a similar argument it is straightforward to see that
 this sequence is also a Cauchy sequence in $C^{k+1,\alpha}$ and so
 $\gamma^N \to \gamma \in C^{k+1,\alpha}$ which by construction satisfies
 \eqref{satisfynonlin}.

\section{Elliptic Estimates}
In this appendix, we collect some well-posedness results from elliptic theory  for the
Dirichlet and Neumann problems. The first is concerning the Dirichlet problem can be found in e.g.
Theorem 6.6 of  \cite{GT15} when $k = 0$ and Problem 6.2 of \cite{GT15} when $k \geq 1$:
\begin{lemma}
  \label{dirichlet}
  Fix $k \geq 2$ and $\alpha\in (0,1)$ and $f \in C^{k-2,\alpha}(D_0)$,
  $g \in C^{k,\alpha}(\pa D_0)$. Let $a^{ij}, b^i, c$ be smooth
  coefficients and set
  \begin{equation}
   L = \sum_{i,j=1}^2 a^{ij}\pa_i\pa_j +\sum_{i = 1}^2 b^i\pa_i.
   \label{}
  \end{equation}
   Suppose that
  the only solution to
  \begin{equation}
   (L + c)v = 0, \qquad v \in H^1_0(D_0)
   \label{}
  \end{equation}
  is $v = 0$.
  Then the Dirichlet problem
 \begin{align}
  (L + c)u &= f, \quad \text{ in } D_0,\\
  u &= g, \quad \text{ on } \pa D_0,
  \label{}
 \end{align}
 has a unique solution $u \in C^{k,\alpha}(D_0)$, and there is
 a constant $C_0 = C_0 (D_0, \|b\|_{k-2,\alpha})$ with
 \begin{equation}
  \|u\|_{k,\alpha}
  \leq C_0 \big(\|f\|_{k-2, \alpha} + |g|_{k,\alpha}\big).
  \label{schauderdirichlet}
 \end{equation}
\end{lemma}
For the Neumann problem, compatibility is also required.
\begin{lemma}
  \label{neumann}
  Fix $k \geq 2$ and $\alpha\in (0,1)$, and $f \in C^{k-2,\alpha}(D_0)$,
  $g \in C^{k-1,\alpha}(\pa D_0)$ satisfying
  \begin{equation}
   \int_{D_0} f = \int_{\partial D_0} g.
   \label{orth}
  \end{equation}
  Then the Neumann problem
 \begin{align}
  \Delta u  &= f, \quad \text{ in } D_0,\\
  \pa_\n u &= g, \quad \text{ on } \pa D_0,
  \label{}
 \end{align}
 has a unique solution $u \in C^{k, \alpha}(D_0)$ and there is a constant
 $C_1 = C_1(D_0)$ with
 \begin{equation}
  \|u\|_{k,\alpha} \leq C_1\big(  \|f\|_{k-2,\alpha} + |g|_{k-1,\alpha}\big).
  \label{}
 \end{equation}
\end{lemma}

\section{Streamline geometry}\label{stream}

In this appendix, we prove some formulae which are useful for our deformation scheme which uses streamline coordinates.  First, we state some relations between the curvature and vorticity of along a given streamline.

\begin{lemma}\label{lemcurv}  Let $\hat{n} =  {\nabla \psi}/|\nabla \psi|$ and $\hat{\tau} = {\nabla^\perp \psi}/|\nabla \psi|$.  The following formulae hold
  \begin{align} \label{formt}
\hat{\tau}\cdot \nabla \otimes \nabla \psi\cdot \hat{\tau}  &=   |\nabla \psi| \kappa,\\  \label{formn}
 \hat{n}\cdot \nabla \otimes \nabla \psi\cdot \hat{n} &=   \Delta \psi -|\nabla \psi|  \kappa,
  \end{align}
  where $\kappa :=  \hat{\tau}\cdot \nabla \hat{n} \cdot \hat{\tau}$ is the curvature of the streamline.
\end{lemma}
\begin{proof}
We being by noticing that
 \be
 \Delta \psi = {\rm tr} \nabla \otimes \nabla \psi = \hat{n}\cdot \nabla \otimes \nabla \psi\cdot \hat{n} +  \hat{\tau}\cdot \nabla \otimes \nabla \psi\cdot \hat{\tau}.
 \ee
Next, a direct calculation gives
 \be
 |\nabla \psi|  \nabla \hat{n}= \nabla \otimes \nabla \psi -( \hat{n}\cdot \nabla \otimes \nabla \psi \cdot \hat{n})  \hat{n}\otimes  \hat{n} - ( \hat{n}\cdot \nabla \otimes \nabla \psi \cdot \hat{\tau})  \hat{n}\otimes  \hat{\tau},
 \ee
so that
$
|\nabla \psi| \hat{\tau}\cdot \nabla \hat{n} \cdot \hat{\tau} =  \hat{\tau}\cdot \nabla \otimes \nabla \psi\cdot \hat{\tau}
$ yielding \eqref{formt} as claimed.
Combining with the above we obtain \eqref{formn}.
\end{proof}

 Before stating the next required lemma, we briefly review action-angle coordinates.  For an in depth discussion, see Arnol'd \cite{A89}, pg  297.
 The streamfunction $\psi$ plays the role of a Hamiltonian for tracer dynamics since $u=\nabla^\perp\psi$.  We assume that the level sets $\{ \psi= c\}$ are simply connected Jordan curves, so that all the integral curves of $u$ (solutions of $\dot{X}= u\circ X$) are periodic orbits.  This system allows for a canonical transformation to action-angle variables, $(x,y)\mapsto (J,\theta)$ which satisfy the following criteria
 \begin{enumerate}
 \item $\psi (x,y)= \Psi(J(x,y))$ for all $(x,y)\in \Omega$ and some function $\Psi$
 \item $\int_{\{\psi= c\}} \rmd \theta = 1$,
 \item $\nabla^\perp J \cdot \nabla \theta=1$.
 \end{enumerate}
Introduce the frequency $\mu^{-1} = \Psi'(J)$.  The phase flow satisfies
\be
\frac{\rmd J}{\rmd t}=0, \qquad
\frac{\rmd \theta}{\rmd t}=\mu^{-1}.
\ee
 The first is simply because the system travels along paths of fixed $J$.  The latter follows from
 \begin{align}
 \frac{\rmd \theta}{\rmd t} &= \frac{\rmd \theta}{\rmd x }\frac{\rmd x}{\rmd t}+  \frac{\rmd \theta}{\rmd y }\frac{\rmd y}{\rmd t}=\nabla^\perp \psi \cdot \nabla \theta=\Psi'(J) (J_x \theta_y-J_y \theta_x )=\mu^{-1} (\nabla^\perp J \cdot \nabla \theta)= \mu^{-1}.
 \end{align}
 The period for each orbit  $\{\psi = c\}$ is  the travel time $\mu:= \mu(c)$ given by
\be
\mu(c) =\oint_{\{ \psi=c\}} \frac{\rmd \ell}{|\nabla \psi|},
\ee
and the line element for each orbit satisfies
 \be
 \rmd \ell = \sqrt{\dot{x}^2 +\dot{y}^2} =  |\nabla \psi| \rmd t = \mu^{-1} |\nabla J| \rmd t = |\nabla J| \rmd \theta.
 \ee

We now give the rule for differentiating functions integrated over streamlines.

\begin{lemma}\label{geolem}
For $f\in C^1(\Omega)$ we have
\be
\frac{\rmd}{\rmd c} \oint_{\{ \psi= c\}}f \frac{\rmd \ell}{|\nabla \psi|} =\oint_{\{ \psi= c\}} \frac{ \nabla \psi \cdot \nabla f - f\left(\omega- 2 \kappa |\nabla \psi| \right)}{|\nabla \psi|^2} \frac{\rmd \ell}{|\nabla \psi|},
\ee
 where $\kappa =  \hat{\tau}\cdot \nabla \hat{n} \cdot \hat{\tau}$ is the curvature of the streamline and $\omega:=\Delta \psi$.
\end{lemma}

\begin{proof}[Proof of Lemma \ref{geolem}]
 First we show that for $g\in C^1(\Omega)$, we have
\be
\frac{\rmd}{\rmd c} \oint_{\{ \psi= c\}}g |\nabla \psi|  \rmd \ell =\oint_{\{ \psi= c\}} \frac{1}{|\nabla \psi|}\Big( \nabla g\cdot \nabla \psi + g \Delta \psi \Big) \rmd \ell.
\ee
 To establish this, set $F:= g\nabla^\perp  \psi$ and $\rmd l = (\dot{x} \rmd t , \dot{y} \rmd t)$.  Then $F\cdot \rmd l = g|\nabla \psi|\rmd \ell$.  By Green's theorem,
 \be
 \oint_{\{ \psi= c\}} F\cdot \rmd l = \iint_{\overline{\{ \psi= c\}}} \nabla^\perp\cdot F  \rmd x\rmd y= \iint_{\overline{\{ \psi= c\}}} [ \nabla g \cdot \nabla \psi+ g\Delta \psi] \rmd x\rmd y.
 \ee
 Then, for two values $c_0\leq c_1$ in the range of $\psi$, we have
 \begin{align}
  \oint_{\{ \psi= c_1\}} F\cdot \rmd l -   \oint_{\{ \psi= c_0\}} F\cdot \rmd \ell &= \iint_{\{ c_0\leq \psi\leq c_c\}}[ \nabla g \cdot \nabla \psi+ g\Delta \psi] \rmd x\rmd y\\
  &= \iint_{\{ c_0\leq \psi\leq c_c\}}[ \nabla g \cdot \nabla \psi+ g\Delta \psi] \rmd \theta \rmd J,
 \end{align}
 where we made a change of variables to action angle coordinates (the Jacobian is unity).  Finally,
\be
\iint h  \rmd \theta \rmd J = \iint h  \mu(\psi) \rmd \theta \rmd \psi= \iint h  \rmd t \rmd \psi =  \iint \frac{h}{|\nabla \psi|} \rmd \ell \rmd \psi,
\ee
for any integrable $h$.
The result follows from taking the coincidence limit $c_1\to c_0$ of the difference quotients.  The lemma then follows by applying the formula with $g= f/|\nabla \psi|^2$.  This gives
\be
\frac{\rmd}{\rmd c} \oint_{\{ \psi= c\}}\frac{f}{ |\nabla \psi|} \rmd \ell =\oint_{\{ \psi= c\}} \frac{ \nabla \psi \cdot \nabla f + f\left(\Delta \psi - 2\widehat{\nabla \psi} \cdot \nabla \otimes \nabla \psi \cdot \widehat{\nabla \psi}\right)}{|\nabla \psi|^3} \rmd \ell.
\ee
To work this into the stated form we appeal to Lemma \ref{lemcurv}.
\end{proof}

 \subsection*{Acknowledgments}   We thank T. M.  Elgindi and H. Q. Nguyen for  insightful discussions.  We thank  T. M.  Elgindi in particular for pointing out Lemma \ref{h1h2h3lem}.   The work of PC was partially supported by NSF grant DMS-1713985 and by the
Simons Center for Hidden Symmetries and Fusion Energy  award \# 601960.
Research of TD was partially supported by
NSF grant DMS-1703997.  Research of DG was partially supported by
the Simons Center for Hidden Symmetries and Fusion Energy.

\end{document}